\documentclass{article}[12pt]
\usepackage[utf8]{inputenc}
\usepackage{authblk}
\usepackage{amsmath, amssymb, amsthm, amsfonts, gensymb, tikz, commath, float, mathtools, enumerate, breqn, circuitikz}
\usetikzlibrary{positioning, arrows, arrows.meta, cd}
\usetikzlibrary{shapes,backgrounds,positioning,petri,topaths,calc}

\usepackage[all]{xy}
\usepackage{tabularx}
\usepackage{multirow}
\usepackage{esvect}
\usepackage{subcaption}
\usepackage{stmaryrd} % For Hirzebruch continued fractions notation

\usepackage[text={15cm,23cm}]{geometry} % Imposta pagina
\usepackage[colorlinks=true,%
            linkcolor=red!50!black,%
            citecolor=blue!50!black,%
            urlcolor=darkgray]{hyperref}  % Needs to go last
\theoremstyle{plain}

\newtheorem{lem}{Lemma}%[section]
\newtheorem{thm}[lem]{Theorem}
\newtheorem{cor}[lem]{Corollary}
\newtheorem{prop}[lem]{Proposition}
\newtheorem{conj}[lem]{Conjecture}

\theoremstyle{definition}
\newtheorem*{rem}{Remark}
\newtheorem{ex}{Example}

\newtheorem{defn}[lem]{Definition}

\newcommand{\qth}{q^{-1}[3]_q}

\newcommand{\Z}{\mathbb{Z}}

\newcommand{\Q}{\mathbb{Q}}

\newcommand{\X}{\mathcal{X}}

\newcommand{\G}{\mathcal{G}}

\newcommand{\Tr}{\textup{Tr}}

\newcommand{\SL}{\mathrm{SL}}
\newcommand{\PSL}{\mathrm{PSL}}
\newcommand{\PGL}{\mathrm{PGL}}

\newcommand{\half}{\frac{1}{2}}

\def\a{\alpha}
\def\b{\beta}

\def\g{\gamma}

\title{On $q$-deformed Markov numbers.
Cohn matrices and\\ perfect matchings with weighted edges}

\author{Sam Evans,
Perrine Jouteur,
Sophie Morier-Genoud,
Valentin Ovsienko}

\date{}

\begin{document}

\maketitle

\begin{abstract}

We consider a natural $q$-deformation of the classical Markov numbers.
This $q$-deformation is closely related to $q$-deformed rational numbers
recently introduced by two of us.
Both notions, those of $q$-rationals and $q$-Markov numbers, are based on
invariance with respect to the action of the modular group $\PSL(2,\Z)$.
We prove that every Markov number has a unique $q$-deformation,
which is a monic unimodal palindromic Laurent polynomial with positive integer coefficients.
The $q$-Markov numbers can be calculated in terms of the traces of $q$-deformed Cohn matrices,
and we show that $q$-Markov numbers are independent of the choice of such matrices.
We construct a combinatorial model counting perfect matchings of snake graphs
with weighted edges.

\end{abstract}

%%%%%%%%%%%%%%%%%%%%
%%%%%%%%%%%%%%%%%%%%
\section{Introduction}
%%%%%%%%%%%%%%%%%%%%
%%%%%%%%%%%%%%%%%%%%

A classical {\it Markov triple} is a triple $(a,b,c)$
of positive integers satisfying the Diophantine equation
\begin{equation}
\label{MarEq}
a^2+b^2+c^2=3abc,
\end{equation}
called the Markov equation.
The  {\it Markov numbers} are the positive integers appearing in a Markov triple.
The sequence of Markov numbers (sequence A002559 in~\cite{OEIS}) starts as follows
$$
1, 2, 5, 13, 29, 34, 89, 169, 194, 233, 433, 610, 985, 1325, \ldots
$$
Discovered in~\cite{Mar}, Markov numbers play crucial role in many branches of mathematics,
from number theory to topology, combinatorics, and mathematical physics.
The celebrated Markov, or Frobenius~\cite{Fro} {\it uniqueness conjecture} affirms that every Markov number
appears exactly once as the maximal number of a Markov triple.
We refer to~\cite{Aig,Reu} for excellent overviews of the status of this conjecture,
and, more generally, for the Markov theory.

The goal of this paper is to study the $q$-analogue of the Markov equation
\begin{eqnarray}
\label{qMarEq}
a_q^2+b_q^2+c_q^2&=&q^{-1}\left[3\right]_qa_qb_qc_q
\;+\;
\left(q - 1\right)(q^{-1} - 1),
\end{eqnarray}
where $\left[3\right]_q=q^2+q+1$ and $q$ is a formal variable.
We will be considering Laurent polynomials in~$q$ with integer coefficients that appear in a triple
$(a_q,b_q,c_q)\in\Z[q^{\pm1}]^3$, satisfying~\eqref{qMarEq}.
We will furthermore assume that substituting $q=1$ we obtain a Markov triple of positive integers:
$$
(a_q,b_q,c_q)\Big\vert_{q=1}=(a,b,c).
$$
In this sense, we understand the triple of polynomials $(a_q,b_q,c_q)$ 
as a $q$-deformation of the Markov triple $(a,b,c)$.
Note that we do not assume a-priori any property of coefficients of Laurent polynomials, 
except that they are integers.

Our study of the equation~\eqref{qMarEq} is motivated, in part, by
a natural appearance of this equation in the context of $q$-deformed rationals.
The solutions to~\eqref{qMarEq} are polynomials with many beautiful properties
such as positivity, palindromicity, and unimodality.
Bearing in mind the idea that a $q$-deformation retains all of the information about the initial object, but often with more detail, it is always interesting to consider natural $q$-analogues.
For instance, the uniqueness conjecture is clearly true in the $q$-deformed case.
This paper is a step in the study of the $q$-analogue of the Markov equation.
We expect further development of the subject and, in particular, more connections to other topics, such as
the theory of dimer coverings.

The first result of this paper is the following existence and uniqueness statement
that will be proved in Section~\ref{UniSec}.

\begin{thm}
\label{ExtUniq}
For every Markov triple $(a,b,c)\in\Z^3_{>0}$
there exists a unique triple of Laurent polynomials with integer coefficients $(a_q,b_q,c_q)\in\Z[q^{\pm1}]^3$ 
which are solutions to~\eqref{qMarEq}
such that $(a_q,b_q,c_q)\vert_{q=1}=(a,b,c)$.
\end{thm}

Proof of Theorem~\ref{ExtUniq} is based on the following recurrent procedure 
that provides one way to calculate the $q$-Markov numbers explicitly.
This recurrence is very similar to the classical case.

\begin{prop}
\label{RecProp}
Every $q$-Markov triple $(a_q,b_q,c_q)$ can be obtained from the triple $(1,1,1)$
by performing a sequence of transformations $(a_q,b_q,c_q)\mapsto(a_q,b_q,c'_q)$, where
\begin{eqnarray}
\label{qMutEq}
c'_q &:=& q^{-1}\left[3\right]_qa_qb_q-c_q,
\end{eqnarray}
combined with permutations of $a_q,b_q,$ and $c_q$.
\end{prop}

Theorem~\ref{ExtUniq} means that every Markov triple $(a,b,c)$,
has a unique $q$-analogue $(a_q,b_q,c_q)$ solution to~\eqref{qMarEq},
that we call {\it the $q$-Markov triple}.
Furthermore, combined with Proposition~\ref{RecProp} and under the assumption that the Markov conjecture holds true,
Theorem~\ref{ExtUniq} implies that every Markov number $m$ has a unique $q$-analogue
that we denote by $m_q$.

The first $q$-deformations of the Markov numbers are as follows
$$
\begin{array}{rcl}
1_q&=&1,\\[2pt]
2_q&=&q+q^{-1},\\[2pt]
5_q&=&q^2+q+1+q^{-1}+q^{-2},\\[2pt]
13_q&=&q^3 + 2q^2 + 2q + 3 + 2q^{-1} + 2q^{-2} + q^{-3}, \\[2pt]
29_q&=&q^4 + 2q^3 + 4q^2 + 5q + 5 + 5q^{-1} + 4q^{-2} + 2q^{-3} + q^{-4},  \\[2pt]
34_q&=&q^4 + 3q^3 + 4q^2 + 6q + 6 + 6q^{-1} + 4q^{-2} + 3q^{-3} + q^{-4},  \\[2pt]
89_q&=&q^5 + 4q^4 + 7q^3 + 11q^2 + 14q + 15 + 14q^{-1} + 11q^{-2} + 7q^{-3} + 4q^{-4} + q^{-5},    \\[2pt]
169_q&=&q^6 + 3q^5 + 8q^4 + 14q^3 + 20q^2 + 25q  + 27 + 25q^{-1} + 20q^{-2}+ 14q^{-3} + 8q^{-4} + 3q^{-5}+ q^{-6}, \\[2pt]
194_q&=&q^6 + 4q^5 + 9q^4 + 16q^3 + 23q^2 + 29q + 30 + 29q^{-1} + 23q^{-2} + 16q^{-3} + 9q^{-4} + 4q^{-5}+ q^{-6}, \\[2pt]
\cdots&&\cdots
\end{array}
$$
The sequence of coefficients of these polynomials does not appear in the OEIS.
The study of combinatorial and geometric properties of this sequence 
seems to be an interesting and challenging open problem.
In Section~\ref{SnakeSec}, we give one such combinatorial interpretation.

It is well-known that the classical Markov numbers are closely related to
modular group $\PSL(2,\Z)$.
The $q$-Markov numbers we consider 
are  related to $q$-deformed rational numbers~\cite{MGOfmsigma}
and the $q$-deformed modular group~\cite{LMGadv}.
This connection is explored in Section~\ref{CohnSec}.
In particular, we calculate explicitly $q$-analogues of the Cohn matrices,
following Aigner's classification~\cite{Aig},
and prove (see Theorem~\ref{CohnShiftThm}) that the $q$-Markov numbers do not depend on the choice of the initial Cohn matrices.

Recall that every Markov number is associated with a so-called {\it snake graph}.
For explanation of this notion; see, e.g.~\cite{Pro,Aig,Sch} and Section~\ref{MarSec}.
A famous result states that Markov numbers are equal to the number of perfect matchings
(or dimer coverings) of the corresponding snake graphs.
For more details; see~\cite[Thm 7.1]{Pro}, and also~\cite[Thm. 7.12]{Aig},\cite[Thm. 5.1]{Sch}.
It is a general rule that the coefficients of $q$-deformed combinatorial objects have the same combinatorial meaning as the initial object, but more refined.
We give a combinatorial model for $q$-Markov numbers in terms of the perfect matchings
of snake graphs with weighted edges. 

Our rule to assign weights to the edges of snake graphs is the following.
The edges at the western and southern borders have weights $q^{-1}$ and $q$ alternately,
all other edges have weight~$1$; see Section~\ref{SnakeSec}.
Taking the product of weights of the edges of every perfect matching, we define
the {\it weighted number} of perfect matchings which is a Laurent polynomial in~$q$.

The main result of the paper is as follows (for a more precise statement; see Theorem~\ref{SnakeThmBis}).

\begin{thm}
\label{SnakeThm}
The $q$-Markov number $m_q$ equals the weighted number of perfect matchings in the snake graph corresponding to~$m$.
\end{thm}

Figure~\ref{SnakeFig1} provides an example of a snake graph with weighted edges.
The total number of perfect matchings of this graph is equal to the Markov number $433$.
Taking the product of weights of the edges of every perfect matching,
one obtains the polynomial $433_q$.

\begin{figure}[H]
	\centering
	\begin{tikzpicture}[scale=0.85]

		\draw[line width=2pt,blue] (0,0)-- node[left]{$q^{-1}$}(0,1);
		\draw[line width=0.7pt] (2,2)-- (2,3);
		\draw[line width=2pt,red] (0,0)-- node[below]{$q$}(1,0);
		\draw[line width=2pt,blue] (1,0)-- node[below]{$q^{-1}$}(2,0);
	        \draw[line width=2pt,red] (2,0)-- node[below]{$q$}(3,0);
		\draw[line width=0.7pt] (0,1)-- (3,1);
		\draw[line width=0.7pt] (1,0)-- (1,1);
		\draw[line width=0.7pt] (2,0)-- (2,1);
		\draw[line width=2pt,red] (2,1)-- node[left]{$q$}(2,2);
		\draw[line width=2pt,blue] (2,2)-- node[left]{$q^{-1}$}(2,3);
		\draw[line width=0.7pt] (2,2)-- (3,2);
		\draw[line width=2pt,blue] (3,2)-- node[below]{$q^{-1}$}(4,2);
		\draw[line width=2pt,red] (4,2)-- node[below]{$q$}(5,2);
		\draw[line width=0.7pt] (3,0)-- (3,3);
		\draw[line width=0.7pt] (2,3)-- (5,3)--(7,3);
                 \draw[line width=0.7pt] (4,2)-- (4,3);
                 \draw[line width=2pt,red] (6,3)--  node[left]{$q$}(6,4);
                 \draw[line width=2pt,blue] (6,4)-- node[left]{$q^{-1}$}(6,5);
                 \draw[line width=0.7pt] (7,2)-- (7,5);
                 \draw[line width=0.7pt] (6,5)-- (9,5);
                 \draw[line width=0.7pt] (6,4)-- (7,4);
                 \draw[line width=2pt,blue] (7,4)-- node[below]{$q^{-1}$}(8,4);
                 \draw[line width=2pt,red] (8,4)-- node[below]{$q$}(9,4);
                 \draw[line width=0.7pt] (7,4)-- (7,5);
                 \draw[line width=0.7pt] (6,4)-- (6,5);
                 \draw[line width=2pt,blue] (5,2)-- node[below]{$q^{-1}$}(6,2);
		\draw[line width=2pt,red] (6,2)-- node[below]{$q$}(7,2);
		
		    \draw[line width=0.7pt] (5,3)-- (5,2);
		        \draw[line width=0.7pt] (6,3)-- (6,2);
		            \draw[line width=0.7pt] (8,4)-- (8,5);
		                \draw[line width=0.7pt] (9,4)-- (9,5);
		
		\draw (0.5, 0.5) node[] {$2$};
		\draw (1.5, 0.5) node[] {$3$};
		\draw (2.5, 0.5) node[] {$5$};
		\draw (2.5, 1.5) node[] {$7$};
		\draw (2.5, 2.5) node[] {$12$};
		\draw (3.5, 2.5) node[] {$17$};
		\draw (4.5, 2.5) node[] {$29$};
		\draw (5.5, 2.5) node[] {$46$};
		\draw (6.5, 2.5) node[] {$75$};
		\draw (6.5, 3.5) node[] {$104$};
		\draw (6.5, 4.5) node[] {$179$};
		\draw (7.5, 4.5) node[] {$254$};
		\draw (8.5, 4.5) node[] {$433$};
		
			\end{tikzpicture}
\caption{A snake graph with weighted edges. Edges coloured in red have weight~$q$ 
and edges coloured in blue have weight~$q^{-1}$, other edges have weigth~$1$.}
\label{SnakeFig1}
\end{figure}

Perfect matchings (or dimer coverings) with weighted edges are
often used in combinatorics and other branches of mathematics and statistical physics.
Ideas that are quite close to our approach were applied in~\cite{Kuo} in the case of 
the Aztec diamond graph.
For a survey; see, e.g.~\cite{Ken}.
It could be an interesting question to ask whether the weight system corresponding to $q$-Markov numbers has some
particular properties.
Let us also mention that weighted perfect matchings of snake graphs were considered in~\cite{BOSZ},
but with weights that are not defined on the edges.

Recall that Markov numbers are naturally presented in a form of a binary tree, and consequently
every Markov number corresponds to a rational $t\in\Q$ such that $0\leq t\leq1$
and is often denoted by $m^t$.
We refer to~\cite{Aig} for the details; see also Section~\ref{MarSec}.
The claim that $m^t$ is uniquely determined by its rational label~$t$
is equivalent to the Markov uniqueness conjecture, see~\cite{Aig}.

As an application of Theorem~\ref{SnakeThm},
we will show in Section~\ref{InjSec} that in the $q$-deformed case the map
$$
t\in\Q_{[0,1]}\mapsto m^t_q
$$
is injective, i.e. every $q$-Markov number $m^t_q$ is uniquely determined by its rational label $t$.
Note that a similar result, but for a different notion of $q$-deformed Markov numbers was obtained in~\cite{LLS}.

\begin{cor}
\label{UniConj}
Two $q$-Markov numbers, $m^t_q$ and $m^{t'}_q$, with $t\not=t'$
are different Laurent polynomials in~$q$.
\end{cor}

\noindent
In fact, we will calculate $t$ explicitly in terms of the coefficients of the Laurent polynomial~$m^t_q$ (Lemma~\ref{RatLem}).

Another application of Theorem~\ref{SnakeThm} concerns the general properties of the Laurent polynomial~$m^t_q$.
The above examples suggest that the Laurent polynomials arising as $q$-Markov numbers
enjoy several nice general properties.
The following statement will be proved in Section~\ref{ShapeS}.

\begin{cor}
\label{PropertThm}
The $q$-Markov numbers are Laurent polynomials with the following properties

(i)
the coefficients of every polynomial are non-negative integers;

(ii)
every such polynomial is monic, i.e. its leading order coefficient equals~$1$;

(iii)
every polynomial representing a $q$-Markov number is palindromic, or reciprocal, i.e.
$$
m^t_{q^{-1}}=m^t_q.
$$
\end{cor}

Combining with our results, the results of~\cite{OgRa,Ogu} imply yet another property of $q$-Markov numbers.

\begin{prop}
\label{PalindProp}
The sequence of coefficients of every $q$-Markov number $m^t_q$, except for $2_q=q+q^{-1}$, is unimodal.
\end{prop}

Let us also add some historical comments on the subject.
The recent notion of $q$-deformed rationals~\cite{MGOfmsigma}
and the related notion of $q$-deformed modular group~\cite{LMGadv}
suggests that Markov numbers should have a natural $q$-deformation.
Different variants of this notion were tackled in~\cite{LaLa,LLS,Kog,Ogu},
all of them based on the choice of a pair of $q$-deformed Cohn matrices.
In~\cite{LaLa,LLS} a version of $q$-deformed Markov numbers was defined as upper right elements
of some special $q$-deformed Cohn matrices, in~\cite{Kog,Ogu} they were defined as their traces.
The first appearance of the equation~\eqref{qMarEq} is due to Kogiso~\cite{Kog},
who considered this equation as a property of $q$-deformed Markov numbers, rather than their definition.

One important distinction between our approach and the previous ones is that we do not fix
specific Cohn matrices to define $q$-Markov numbers.
The notion of $q$-Markov numbers we study 
coincides (up to a multiple) with that of~\cite{Kog} who considered traces of some particular $q$-deformed Cohn matrices,
and it is different from the one considered in~\cite{LaLa,LLS}, where
$q$-Markov numbers were understood as the upper right element of some particular Cohn matrices. 

The paper is organized as follows.
In Section~\ref{MarSec}, we present well-known results about Markov numbers.
This allows us to introduce the notation and the main ideas of the theory.
Section~\ref{UniSec} contains the proof of the fact that every Markov number~$m^t$ has a canonical $q$-analogue~$m^t_q$.
In Section~\ref{CohnSec}, we introduce the $q$-deformed action
of the modular group $\PSL(2,\Z)$ and calculate $q$-deformed Cohn matrices.
In Section~\ref{SnakeSec}, we define weighted snake graphs and prove our main result.
Finally, in Section~\ref{ExtrSec}, we collect a number of observations, conjectures and open questions.
In particular, we try to outline possible relations of $q$-Markov numbers to cluster algebra.

%%%%%%%%%%%%%%%%%%%%
%%%%%%%%%%%%%%%%%%%%
\section{A compendium on Markov numbers}\label{MarSec}
%%%%%%%%%%%%%%%%%%%%
%%%%%%%%%%%%%%%%%%%%

In this section, we collect well-known facts about the classical Markov numbers
that will be relevant for this paper.

%%%%%%%%%%%%%%%%%%%%
\subsection{A recurrent procedure}
%%%%%%%%%%%%%%%%%%%%

Markov proved that all positive integer solutions of the equation~\eqref{MarEq}
can be obtained from the simplest Markov triple $(a,b,c)=(1,1,1)$
by performing transformations $(a,b,c)\mapsto(a,b,c')$, where
\begin{equation}
\label{MutEq}
c'
\quad:=\quad
3ab-c
\quad=\quad
\frac{a^2+b^2}{c},
\end{equation}
combined with permutations of $a,b$, and $c$.
These transformations are often called ``mutations'' because of the connection to cluster algebra; see Section~\ref{CluSec}.
Note that both expressions in~\eqref{MutEq} are important, 
the first implies that $c'$ remains integer, while the second equality insures that $c'$ is positive.
An elementary proof of this statement can be found in~\cite[p.46]{Aig}.

%%%%%%%%%%%%%%%%%%%%
\subsection{The Markov tree}
%%%%%%%%%%%%%%%%%%%%

The Markov numbers can be organized with the help of an infinite binary tree,
sometimes called the Conway topograph; see~\cite{BV}. 
The edges of the  infinite binary tree have an orientation so that at each vertex there are one ingoing and two outgoing arrows.
In our illustrations some edges are coloured in green indicating that there is a choice for their orientation but the choice does not affect the upper part of the tree, 
see Section \ref{Mirror} for further discussions.
\begin{figure}[H]
\begin{center}
\begin{tikzpicture}[scale=0.7, every node/.style={scale=0.75}]
    \node at (1.75,1.25) {$2$};
    \node at (-1.75,1.25){$1$};
    \node at (0,4) {$5$};
    \node at (3.11, 3.8) {$29$};
    \node at (2.20, 5.76) {$433$}; 
    \node at (4.37, 1.98) {$169$}; 
    \node at (0.7, 6.3) {$6466$}; 
    \node at (3.8, 5.85) {$37666$}; 
    \node at (5.25, 3.35) {$14701$}; 
    \node at (4.2, 0.5) {$985$}; 
    \node at (-3.11, 3.8) {$13$};
    \node at (-2.20, 5.76) {$194$}; 
    \node at (-4.37, 1.98) {$34$}; 
    \node at (-0.7, 6.3) {$2897$}; 
    \node at (-3.8, 5.85) {$7561$}; 
    \node at (-5.25, 3.35) {$1325$}; 
    \node at (-4.2, 0.5) {$89$}; 
    
        \node at (0,-1){$1$};
            \node at (-2.125,-1.125){$1$};

%\draw[thick, ->] (0, -0.5) -- (0, 0.75);
   % \draw[thick] (0, 0.75) -- (0, 2);
    \draw[thick, ->, >=stealth] (0, 0.5) -- (0, 2);
    
        \draw[thick,  ->, >=stealth] (-1.5, -0.5)-- (0, 0.5);
        
           \draw[thick,  ->, >=stealth] (0, 0.5) -- (1.5, -0.5);
           
          \draw[thick, green] (-1.5,-0.5) -- (-2.75, -0.5);
          
             \draw[thick, green] (-1.5,-0.5) -- (-1.5, -1.75);

    \draw[thick,  ->, >=stealth] (0,2) -- (1.73,3);
        
        \draw[thick,  ->, >=stealth] (1.73, 3) -- (1.99, 4.48);	

            \draw[thick,  ->, >=stealth] (1.99, 4.48) -- (1.31, 5.53);
            
                \draw[thick,  ->, >=stealth] (1.31, 5.53) -- (0.36, 5.87);
                \draw[thick,  ->, >=stealth] (1.31, 5.53) --  (1.23, 6.53);
                
            \draw[thick,  ->, >=stealth] (1.99, 4.48) -- (2.98, 5.24);	
            
                \draw[thick,  ->, >=stealth] (2.98, 5.24) -- (3.24, 6.21);
                \draw[thick,  ->, >=stealth] (2.98, 5.24) -- (3.98, 5.24);

        \draw[thick,  ->, >=stealth] (1.73, 3) -- (3.14, 2.49);

            \draw[thick,  ->, >=stealth] (3.14, 2.49) -- (4.29, 2.97);

                \draw[thick,  ->, >=stealth] (4.29, 2.97) -- (4.79, 3.84);
                \draw[thick,  ->, >=stealth] (4.29, 2.97) -- (5.26, 2.71);
            
            \draw[thick,  ->, >=stealth] (3.14, 2.49) -- (3.72, 1.38);
                
                \draw[thick,  ->, >=stealth] (3.72, 1.38) -- (4.63, 0.96);
                \draw[thick,  ->, >=stealth] (3.72, 1.38) -- (3.55, 0.39);
                
    \draw[thick,  ->, >=stealth] (0,2) -- (-1.73,3);
    
        \draw[thick,  ->, >=stealth] (-1.73, 3) -- (-1.99, 4.48);	
    
            \draw[thick,  ->, >=stealth] (-1.99, 4.48) -- (-1.31, 5.53);
    
                \draw[thick,  ->, >=stealth] (-1.31, 5.53) -- (-0.36, 5.87);
                \draw[thick,  ->, >=stealth] (-1.31, 5.53) -- (-1.23, 6.53);
    
            \draw[thick,  ->, >=stealth] (-1.99, 4.48) -- (-2.98, 5.24);	
    
                \draw[thick,  ->, >=stealth] (-2.98, 5.24) -- (-3.24, 6.21);
                \draw[thick,  ->, >=stealth] (-2.98, 5.24) -- (-3.98, 5.24);
    
        \draw[thick,  ->, >=stealth] (-1.73, 3) -- (-3.14, 2.49);
    
            \draw[thick,  ->, >=stealth] (-3.14, 2.49) -- (-4.29, 2.97);
    
                \draw[thick,  ->, >=stealth] (-4.29, 2.97) -- (-4.79, 3.84);
                \draw[thick,  ->, >=stealth] (-4.29, 2.97) -- (-5.26, 2.71);
    
            \draw[thick,  ->, >=stealth] (-3.14, 2.49) -- (-3.72, 1.38);
    
                \draw[thick,  ->, >=stealth] (-3.72, 1.38) -- (-4.63, 0.96);
                \draw[thick,  ->, >=stealth] (-3.72, 1.38) -- (-3.55, 0.39);

\end{tikzpicture}
\caption{Markov numbers represented on the Conway topograph.}
\label{fig:Markov Conway}
\end{center}
\end{figure}
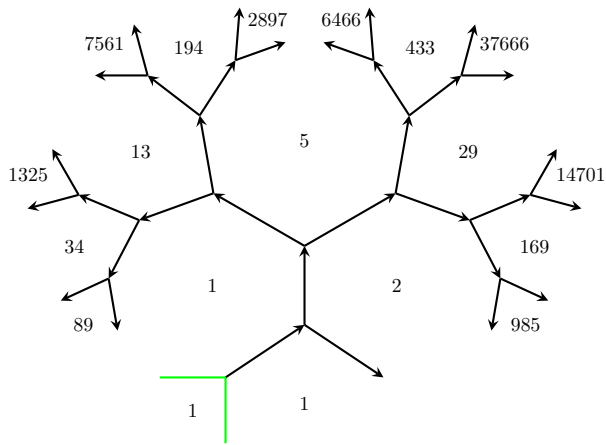

The tree is drawn in the plane cutting it into infinitely many regions,
and every region is labelled by a Markov number.
Three numbers around a vertex constitute
a Markov triple,
and the transformations~\eqref{MutEq} correspond to the following ``fishbone branchings'':
\begin{figure}[H]
\begin{center}
\begin{tikzpicture}[scale=0.8]
 \draw[thick,  ->, >=stealth] (0,0.2) -- (1,1);	
  \draw[thick,  ->, >=stealth] (1,1)-- (2,0.2);	
    \draw[thick,  ->, >=stealth] (1,1)-- (1,2.2);	
       \draw[thick,  ->, >=stealth] (1,2.2)-- (0,3);	
          \draw[thick,  ->, >=stealth] (1,2.2)-- (2,3);
      \node at (1,0.3) {$c$}; 
         \node at (0.4, 1.6) {$a$}; 
          \node at (1.6, 1.6) {$b$}; 
          \node at (1, 2.9) {$c'$}; 
\end{tikzpicture}
\caption{The fishbone}
\label{fig:fishMarkoff}
\end{center}
\end{figure}
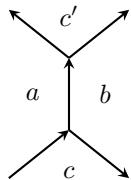

The tree of Markov numbers can be compared with the yet more ancient and better known Farey tree;
see, e.g.~\cite{HW,Aig,Reu}.

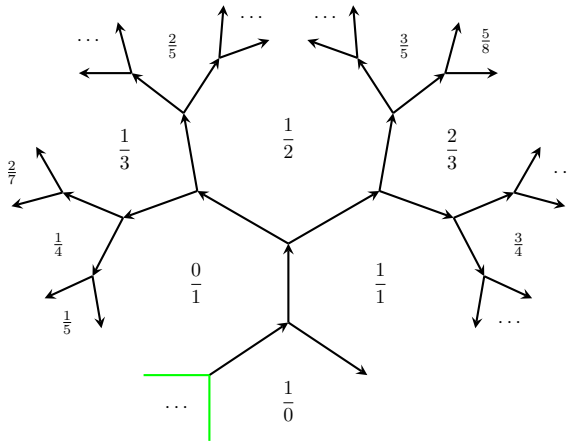
\begin{figure}[H]
\begin{center}
\begin{tikzpicture}[scale=0.7, every node/.style={scale=0.75}]
    \node at (1.75,1.25) {$\dfrac11$};
    \node at (-1.75,1.25){$\dfrac01$};
    \node at (0,4) {$\dfrac12$};
    \node at (3.11, 3.8) {$\dfrac23$};
    \node at  (2.20, 5.76) {$\frac35$}; 
    \node at (4.37, 1.98) {$\frac34$}; 
    \node at (0.7, 6.3) {$\cdots$}; 
   \node at (3.75, 5.85)  {$\frac58$}; 
    \node at (5.25, 3.35) {$\cdots$}; 
    \node at (4.2, 0.5) {$\cdots$}; 
    \node at (-3.11, 3.8) {$\dfrac13$};
    \node at (-2.20, 5.76) {$\frac25$}; 
    \node at (-4.37, 1.98) {$\frac14$}; 
    \node at (-0.7, 6.3) {$\cdots$}; 
    \node at (-3.8, 5.85) {$\cdots$}; 
    \node at (-5.25, 3.35)  {$\frac27$}; 
    \node at (-4.2, 0.5) {$\frac15$}; 
    
        \node at (0,-1){$\dfrac10$};
            \node at (-2.125,-1.125){$\cdots$};

%\draw[thick, ->] (0, -0.5) -- (0, 0.75);
   % \draw[thick] (0, 0.75) -- (0, 2);
     \draw[thick, ->, >=stealth] (0, 0.5) -- (0, 2);
    
        \draw[thick,  ->, >=stealth] (-1.5, -0.5)-- (0, 0.5);
        
           \draw[thick,  ->, >=stealth] (0, 0.5) -- (1.5, -0.5);
           
          \draw[thick, green] (-1.5,-0.5) -- (-2.75, -0.5);
          
             \draw[thick, green] (-1.5,-0.5) -- (-1.5, -1.75);

    \draw[thick,  ->, >=stealth] (0,2) -- (1.73,3);
        
        \draw[thick,  ->, >=stealth] (1.73, 3) -- (1.99, 4.48);	

            \draw[thick,  ->, >=stealth] (1.99, 4.48) -- (1.31, 5.53);
            
                \draw[thick,  ->, >=stealth] (1.31, 5.53) -- (0.36, 5.87);
                \draw[thick,  ->, >=stealth] (1.31, 5.53) --  (1.23, 6.53);
                
            \draw[thick,  ->, >=stealth] (1.99, 4.48) -- (2.98, 5.24);	
            
                \draw[thick,  ->, >=stealth] (2.98, 5.24) -- (3.24, 6.21);
                \draw[thick,  ->, >=stealth] (2.98, 5.24) -- (3.98, 5.24);

        \draw[thick,  ->, >=stealth] (1.73, 3) -- (3.14, 2.49);

            \draw[thick,  ->, >=stealth] (3.14, 2.49) -- (4.29, 2.97);

                \draw[thick,  ->, >=stealth] (4.29, 2.97) -- (4.79, 3.84);
                \draw[thick,  ->, >=stealth] (4.29, 2.97) -- (5.26, 2.71);
            
            \draw[thick,  ->, >=stealth] (3.14, 2.49) -- (3.72, 1.38);
                
                \draw[thick,  ->, >=stealth] (3.72, 1.38) -- (4.63, 0.96);
                \draw[thick,  ->, >=stealth] (3.72, 1.38) -- (3.55, 0.39);
                
    \draw[thick,  ->, >=stealth] (0,2) -- (-1.73,3);
    
        \draw[thick,  ->, >=stealth] (-1.73, 3) -- (-1.99, 4.48);	
    
            \draw[thick,  ->, >=stealth] (-1.99, 4.48) -- (-1.31, 5.53);
    
                \draw[thick,  ->, >=stealth] (-1.31, 5.53) -- (-0.36, 5.87);
                \draw[thick,  ->, >=stealth] (-1.31, 5.53) -- (-1.23, 6.53);
    
            \draw[thick,  ->, >=stealth] (-1.99, 4.48) -- (-2.98, 5.24);	
    
                \draw[thick,  ->, >=stealth] (-2.98, 5.24) -- (-3.24, 6.21);
                \draw[thick,  ->, >=stealth] (-2.98, 5.24) -- (-3.98, 5.24);
    
        \draw[thick,  ->, >=stealth] (-1.73, 3) -- (-3.14, 2.49);
    
            \draw[thick,  ->, >=stealth] (-3.14, 2.49) -- (-4.29, 2.97);
    
                \draw[thick,  ->, >=stealth] (-4.29, 2.97) -- (-4.79, 3.84);
                \draw[thick,  ->, >=stealth] (-4.29, 2.97) -- (-5.26, 2.71);
    
            \draw[thick,  ->, >=stealth] (-3.14, 2.49) -- (-3.72, 1.38);
    
                \draw[thick,  ->, >=stealth] (-3.72, 1.38) -- (-4.63, 0.96);
                \draw[thick,  ->, >=stealth] (-3.72, 1.38) -- (-3.55, 0.39);
\end{tikzpicture}
\caption{The Farey tree of rational numbers.}
\label{fig:Rational tree}
\end{center}
\end{figure}

Rational numbers label regions bounded by the binary tree, and obey the following local rule:
every two rationals, $\frac{r}{s}$ and $\frac{r'}{s'}$, separated by an edge produce a third rational
$$
\frac{r}{s}\oplus\frac{r'}{s'}:=
\frac{r+r'}{s+s'}
$$
called the mediant, or the Farey sum of $\frac{r}{s}$ and $\frac{r'}{s'}$.
\begin{figure}[H]
\begin{center}
\begin{tikzpicture}[scale=0.8]
% \draw[thick,  ->, >=stealth] (0,0.2) -- (1,1);	
  %\draw[thick,  ->, >=stealth] (1,1)-- (2,0.2);	
    \draw[thick,  ->, >=stealth] (1,1)-- (1,2.2);	
       \draw[thick,  ->, >=stealth] (1,2.2)-- (0,3);	
          \draw[thick,  ->, >=stealth] (1,2.2)-- (2,3);
  %    \node at (1,0.3) {$c$}; 
         \node at (0.4, 1.6) {$\frac{r}{s}$}; 
          \node at (1.6, 1.6) {$\frac{r'}{s'}$}; 
          \node at (1, 2.9) {$\frac{r+r'}{s+s'}$}; 
\end{tikzpicture}
\caption{The Farey sum.}
\label{fig:FareyF}
\end{center}
\end{figure}

It is known that
the Farey tree contains all rational numbers $\Q\cup\left\{\frac10\right\}$,
and every rational appears exactly once.
Identifying the trees, one then equips every Markov number with a rational label.
In other words, one has a map
$$
t\in\Q_{[0,1]}\mapsto m^t
$$
widely believed to be injective,
although this is the Markov uniqueness conjecture; see~\cite{Aig}.

%%%%%%%%%%%%%%%%%%%%
\subsection{Cohn matrices}\label{CohnM}
%%%%%%%%%%%%%%%%%%%%

There exists an amazing way of computing Markov numbers using $2\times2$ unimodular integer matrices.
It was discovered by Harvey Cohn~\cite{Coh1}.
An element $C\in\SL(2,\Z)$
$$
C=\begin{pmatrix}
a&b\\[2pt]
c&d
\end{pmatrix},
\qquad a,b,c,d\in\Z,
\quad
ad-bc=1
$$
 is called a {\it Cohn matrix} if 
 $b$ is a Markov number and $\Tr(C)=a+d=3b$.

A classification of Cohn matrices is given in~\cite[Thm. 4.8]{Aig},
all such matrices are parametrized by one integer parameter $n\in\Z$
and can be obtained by multiplying two initial matrices~$A(n)$ and $B(n)$
\begin{equation}
\label{Cohn}
A(n)=
\begin{pmatrix}
n&1\\[2pt]
3n-n^2 -1&3-n
\end{pmatrix},
\qquad\qquad
B(n)=
\begin{pmatrix}
2n+1&2\\[2pt]
-2n^2 +4n+2&5-2n
\end{pmatrix}.
\end{equation}
Observe that 
$$
B(n)=A(n)A(n+1).
$$
The construction is again based on a binary tree:
\begin{figure}[H]
\begin{center}
\begin{tikzpicture}[scale=0.7, every node/.style={scale=0.75}]
    \node at (1.75,1.25) {$B$};
    \node at (-1.75,1.25){$A$};
    \node at (0,4) {$AB$};
    \node at (3.11, 3.8) {$AB^2$};
    \draw (2.20, 5.76) node[scale=1] {$ABAB^2$}; 
    \node at (4.37, 1.98) {$AB^3$}; 
    \node at (0.7, 6.3) {$\cdots$}; 
   \draw (4.5, 5.85) node[scale=1] {$ABAB^2AB^2$}; 
    \node at (5.4, 3.35) {$\cdots$}; 
    \node at (4.2, 0.5) {$\cdots$}; 
    \node at (-3.11, 3.8) {$A^2B$};
    \node at (-2.20, 5.76) {$A^2BAB$}; 
    \node at (-4.37, 1.98) {$A^3B$}; 
    \node at (-0.7, 6.3) {$\cdots$}; 
    \node at (-3.8, 5.85) {$\cdots$}; 
    \draw (-5.4, 3.35) node[scale=1]  {$A^3BA^2B$}; 
    \node at (-4.2, 0.5) {$A^4B$}; 
    
        \node at (0,-1){$A^{-1}B$};
            \node at (-2.125,-1.125){$\cdots$};

%\draw[thick, ->] (0, -0.5) -- (0, 0.75);
   % \draw[thick] (0, 0.75) -- (0, 2);
     \draw[thick, ->, >=stealth] (0, 0.5) -- (0, 2);
    
        \draw[thick,  ->, >=stealth] (-1.5, -0.5)-- (0, 0.5);
        
           \draw[thick,  ->, >=stealth] (0, 0.5) -- (1.5, -0.5);
           
          \draw[thick, green] (-1.5,-0.5) -- (-2.75, -0.5);
          
             \draw[thick, green] (-1.5,-0.5) -- (-1.5, -1.75);

    \draw[thick,  ->, >=stealth] (0,2) -- (1.73,3);
        
        \draw[thick,  ->, >=stealth] (1.73, 3) -- (1.99, 4.48);	

            \draw[thick,  ->, >=stealth] (1.99, 4.48) -- (1.31, 5.53);
            
                \draw[thick,  ->, >=stealth] (1.31, 5.53) -- (0.36, 5.87);
                \draw[thick,  ->, >=stealth] (1.31, 5.53) --  (1.23, 6.53);
                
            \draw[thick,  ->, >=stealth] (1.99, 4.48) -- (2.98, 5.24);	
            
                \draw[thick,  ->, >=stealth] (2.98, 5.24) -- (3.24, 6.21);
                \draw[thick,  ->, >=stealth] (2.98, 5.24) -- (3.98, 5.24);

        \draw[thick,  ->, >=stealth] (1.73, 3) -- (3.14, 2.49);

            \draw[thick,  ->, >=stealth] (3.14, 2.49) -- (4.29, 2.97);

                \draw[thick,  ->, >=stealth] (4.29, 2.97) -- (4.79, 3.84);
                \draw[thick,  ->, >=stealth] (4.29, 2.97) -- (5.26, 2.71);
            
            \draw[thick,  ->, >=stealth] (3.14, 2.49) -- (3.72, 1.38);
                
                \draw[thick,  ->, >=stealth] (3.72, 1.38) -- (4.63, 0.96);
                \draw[thick,  ->, >=stealth] (3.72, 1.38) -- (3.55, 0.39);
                
    \draw[thick,  ->, >=stealth] (0,2) -- (-1.73,3);
    
        \draw[thick,  ->, >=stealth] (-1.73, 3) -- (-1.99, 4.48);	
    
            \draw[thick,  ->, >=stealth] (-1.99, 4.48) -- (-1.31, 5.53);
    
                \draw[thick,  ->, >=stealth] (-1.31, 5.53) -- (-0.36, 5.87);
                \draw[thick,  ->, >=stealth] (-1.31, 5.53) -- (-1.23, 6.53);
    
            \draw[thick,  ->, >=stealth] (-1.99, 4.48) -- (-2.98, 5.24);	
    
                \draw[thick,  ->, >=stealth] (-2.98, 5.24) -- (-3.24, 6.21);
                \draw[thick,  ->, >=stealth] (-2.98, 5.24) -- (-3.98, 5.24);
    
        \draw[thick,  ->, >=stealth] (-1.73, 3) -- (-3.14, 2.49);
    
            \draw[thick,  ->, >=stealth] (-3.14, 2.49) -- (-4.29, 2.97);
    
                \draw[thick,  ->, >=stealth] (-4.29, 2.97) -- (-4.79, 3.84);
                \draw[thick,  ->, >=stealth] (-4.29, 2.97) -- (-5.26, 2.71);
    
            \draw[thick,  ->, >=stealth] (-3.14, 2.49) -- (-3.72, 1.38);
    
                \draw[thick,  ->, >=stealth] (-3.72, 1.38) -- (-4.63, 0.96);
                \draw[thick,  ->, >=stealth] (-3.72, 1.38) -- (-3.55, 0.39);
\end{tikzpicture}
\caption{The Cohn tree.}
\label{fig:Christoffel tree AB}
\end{center}
\end{figure}

The following statement is due to Harvey Cohn~\cite{Coh1},
its elegant proof can be found in~\cite{Aig}.

Once again, identifying the binary trees, every matrix in the Cohn tree can be labelled by a rational number.
The notation due to Aigner is $C_t$.
The tree is determined by $C_{\frac01}=A(n)$ and $C_{\frac11}=B(n)$, for some fixed $n$, that are called \textit{initial Cohn matrices}.

\begin{thm}[\cite{Coh1,Aig}]
\label{CohnThm}
Every matrix $C_t$ appearing in this tree is a Cohn matrix,
its upper right element is precisely the corresponding Markov number~$m^t$,
while its trace equals~$3m^t$.
\end{thm}

The most common choice of the Cohn matrices~\eqref{Cohn} corresponds to
$n=0,1$, or $2$:
\begin{eqnarray*}
A(0)=
\begin{pmatrix}
0&1\\[2pt]
-1&3
\end{pmatrix},
&
B(0)
=
\begin{pmatrix}
1&2\\[2pt]
2& 5
\end{pmatrix},
\\[10pt]
A(1)=
\begin{pmatrix}
1&1\\[2pt]
1&2
\end{pmatrix},
&
B(1)
=
\begin{pmatrix}
  3 & 2 \\[2pt]
4 & 3
\end{pmatrix},
\\[10pt]
A(2)=
\begin{pmatrix}
2&1\\[2pt]
1&1
\end{pmatrix},
&
B(2)
=
 \begin{pmatrix}
 5& 2 \\[2pt]
2 & 1
\end{pmatrix}.
\end{eqnarray*}

\begin{rem}
Markov numbers appear simultaneously as traces and as upper right elements of Cohn matrices.
However, the notion of trace of a matrix is much more fundamental than its particular entry.
This obvious reasoning becomes more important in the $q$-deformed case.
We will see that the trace of $q$-deformed Cohn matrices are independent of their choice, 
while the upper right elements are different.
\end{rem}

%%%%%%%%%%%%%%%%%%%%
\subsection{Snake graphs}
%%%%%%%%%%%%%%%%%%%%

A {\it snake graph} is a particular connected graph in the plane
that has only vertical and horizontal edges.
It can be obtained by gluing elementary square boxes along edges.
A new box is attached to the previous one at the right or at the top.

Snake graphs related to Markoff numbers appear in Harvey Cohn's paper~\cite[Fig 1.]{Coh} and in Propp's paper \cite{Pro}.
Choosing a rational number $t=\frac{k}{s}\in[0,1]$, one draws in the plane a line of slope $t $ and one picks all of the $1\times1$ squares 
in the $\Z^2$-grid crossed by the diagonal of the $k\times s$ rectangle.
The resulting snake graph is called in~\cite{Coh} the ``polygonal approximation'' of the diagonal of the $k\times s$ rectangular.
We reproduce here (see Fig.~\ref{fig:snakes} left) a somewhat transposed figure from Cohn's paper
(see Fig. 1 in~\cite{Coh}) for $t=\frac35$.

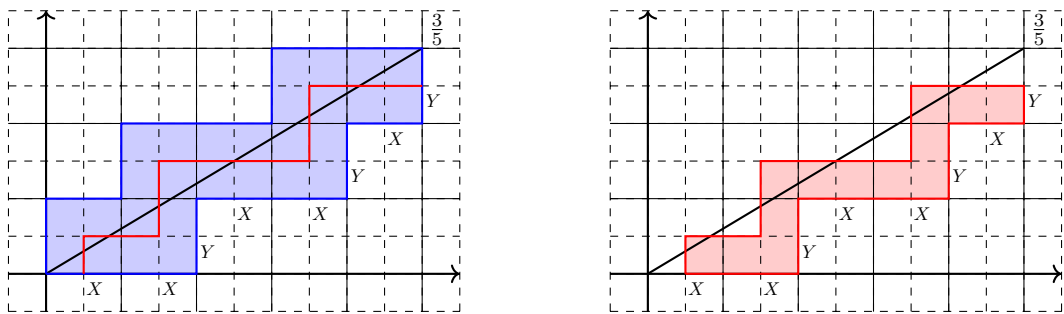
\begin{figure}[H]
\begin{center}
\begin{tikzpicture}[scale=1]
  \draw[color=blue!20, fill=blue!20] (1,1) rectangle (3,2);
    \draw[color=blue!20, fill=blue!20] (2,2) rectangle (5,3);
      \draw[color=blue!20, fill=blue!20] (4,3) rectangle (6,4);
\draw[step=1.0,black] (0.5,0.5) grid (6.5,4.5);
\draw[step=0.5,black, dashed] (0.5,0.5) grid (6.5,4.5);
\draw[thick, ->] (1,0.5) -- (1,4.5);
\draw[thick, ->] (0.5,1) -- (6.5,1);
\draw[thick ] (1,1) -- (6,4);
\draw[thick, blue ] (1,1) -- (1,2) -- (2,2) -- (2,3) -- (4, 3) -- (4,4) -- (6,4) -- (6,3) -- (5,3) -- (5,2) -- (3,2) -- (3,1) -- (1,1);
\draw[thick, red ] (1.5,1) -- (1.5, 1.5) -- (2.5,1.5) -- (2.5,2.5) -- (4.5, 2.5) -- (4.5,3.5) -- (6,3.5) ;
      \node at (6.2, 4.25){$\frac35$};
      \draw (1.65, 0.8) node[scale=0.7] {$X$}; 
        \draw (2.65, 0.8) node[scale=0.7] {$X$}; 
          \draw (3.65, 1.8) node[scale=0.7] {$X$}; 
            \draw (4.65, 1.8) node[scale=0.7] {$X$}; 
              \draw (5.65, 2.8) node[scale=0.7] {$X$}; 
      \draw (3.15,1.3) node[scale=0.7] {$Y$}; 
      \draw (5.15,2.3) node[scale=0.7] {$Y$}; 
	\draw (6.15,3.3) node[scale=0.7] {$Y$}; 

%\draw[step=1.0,black,thin, dashed, xshift=0.5cm,yshift=0.5cm] (0.5,0.5) grid (5.5,4.5);
    \begin{scope}[shift={(8,0)}]
     \draw[color=red!20, fill=red!20] (4.5,2.5) rectangle (5,3);
       \draw[color=red!20, fill=red!20] (2.5,1.5) rectangle (3,2);
       \draw[color=red!20, fill=red!20] (4.5,3) rectangle (6,3.5);
      \draw[color=red!20, fill=red!20] (2.5,2) rectangle (5, 2.5);
        \draw[color=red!20, fill=red!20] (1.5,1) rectangle (3,1.5);
    \draw[step=1.0,black] (0.5,0.5) grid (6.5,4.5);
\draw[step=0.5,black, dashed] (0.5,0.5) grid (6.5,4.5);
\draw[thick, ->] (1,0.5) -- (1,4.5);
\draw[thick, ->] (0.5,1) -- (6.5,1);
\draw[thick ] (1,1) -- (6,4);
\draw[thick, red ] (1.5,1) -- (1.5, 1.5) -- (2.5,1.5) -- (2.5,2.5) -- (4.5, 2.5) -- (4.5,3.5) -- (6,3.5) -- (6,3) -- (5,3) -- (5,2) -- (3,2) -- (3,1) -- (1.5,1);
     \node at (6.2, 4.25){$\frac35$};
      \draw (1.65, 0.8) node[scale=0.7] {$X$}; 
        \draw (2.65, 0.8) node[scale=0.7] {$X$}; 
          \draw (3.65, 1.8) node[scale=0.7] {$X$}; 
            \draw (4.65, 1.8) node[scale=0.7] {$X$}; 
              \draw (5.65, 2.8) node[scale=0.7] {$X$}; 
      \draw (3.15,1.3) node[scale=0.7] {$Y$}; 
      \draw (5.15,2.3) node[scale=0.7] {$Y$}; 
	\draw (6.15,3.3) node[scale=0.7] {$Y$}; 
    \end{scope}
\end{tikzpicture}
\caption{Snake graphs. The Cohn snake (left) {\it vs} the domino snake $\mathcal{G}_{t}$ (right).}
\label{fig:snakes}
\end{center}
\end{figure}

The construction of the snake graph that will be relevant for us refines that of Cohn.
Tracing the ``median zig-zag line'' joining the point $(\half,0)$ and $(s,k-\half)$
(see the red line on Fig.~\ref{fig:snakes}, left) and collecting the $\half\times\half$ square boxes under the line,
one obtains the snake graph that we denote by~$\mathcal{G}_{t}$.
Note that $\mathcal{G}_{t}$ consists of $\half \times \half$ boxes  and thus contains more boxes 
than Cohn's polygonal approximation.
The snake graph $\G_{\frac35}$ in Fig.~\ref{fig:snakes} (right) is precisely that of our 
previous example depicted in Fig.~\ref{SnakeFig1}.
The snake graph $\G_t$ is called the ``domino graph'' in~\cite{Aig}.

Let us mention that snake graphs are useful in several branches of combinatorics connected to continued fractions,
cluster algebras, and many other topics; see~\cite{MSW,CR} and references therein.
A systematic study of arbitrary snake graphs was initiated by Canakci and Schiffler in~\cite{CS}.
For a survey; see~\cite{Sch}.

Recall that a {\it perfect matching} of a graph is a subset of its edges
such that every vertex of the graph is incident to exactly one edge of the matching.
Perfect matchings in simplest snake graphs
are collected in Fig.~\ref{PerFig} .

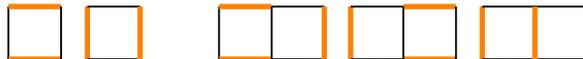
\begin{figure}[H]
	\centering
	\begin{tikzpicture}[scale=0.70]
		
		\draw[line width=0.7pt] (0,0)--(0,1);
		\draw[line width=2pt,orange] (0,0)-- (1,0);
		\draw[line width=2pt,orange] (0,1)-- (1,1);
		\draw[line width=0.7pt] (1,0)-- (1,1);
		
		\draw[line width=2pt,orange] (1.5,0)--(1.5,1);
		\draw[line width=0.7pt] (1.5,0)-- (2.5,0);
		\draw[line width=0.7pt] (1.5,1)-- (2.5,1);
		\draw[line width=2pt,orange] (2.5,0)-- (2.5,1);
		
		\draw[line width=0.7pt] (4,0)--(4,1);
		\draw[line width=02pt,orange] (4,1)--(5,1);
		\draw[line width=02pt,orange] (4,0)--(5,0);
		\draw[line width=0.7pt] (5,0)--(5,1);
		\draw[line width=0.7pt] (5,1)--(6,1);
		\draw[line width=0.7pt] (5,0)--(6,0);
		\draw[line width=02pt,orange] (6,0)--(6,1);

		\draw[line width=02pt,orange] (6.5,0)--(6.5,1);
		\draw[line width=0.7pt] (6.5,1)--(7.5,1);
		\draw[line width=0.7pt] (6.5,0)--(7.5,0);
		\draw[line width=02pt,orange] (7.5,0)--(8.5,0);
		\draw[line width=02pt,orange] (7.5,1)--(8.5,1);
		\draw[line width=0.7pt] (7.5,0)--(7.5,1);
		\draw[line width=0.7pt] (8.5,0)--(8.5,1);

		\draw[line width=02pt,orange] (9,0)--(9,1);
		\draw[line width=0.7pt] (9,1)--(10,1);
		\draw[line width=0.7pt] (9,0)--(10,0);
		\draw[line width=0.7pt] (10,0)--(11,0);
		\draw[line width=0.7pt] (10,1)--(11,1);
		\draw[line width=02pt,orange] (10,0)--(10,1);
		\draw[line width=02pt,orange] (11,0)--(11,1);

			\end{tikzpicture}
\caption{Examples of perfect matchings. There are two perfect matchings for a single box and three for a two-box snake.}
\label{PerFig}
\end{figure}

The following statement is a well-known fundamental result explaining the combinatorial nature of Markov numbers.

\begin{thm}
\label{ProffThm}
Every Markov number $m^t$ is equal to the number of perfect matchings in~$\mathcal{G}_t$.
\end{thm}

This result first appears in~\cite[Thm. 7.1]{Pro}.
A self-contained transparent proof can be found in~\cite[Thm. 7.12]{Aig}.
The statement can be deduced from the result of~\cite{CR} about continued fractions.
Note that the result of~\cite{CR} is actually more general.
Another approach based on connections with cluster algebras is outlined in~\cite{Sch}.

The main goal of this paper is to give a generalization of Theorem~\ref{ProffThm}.
For this end we will assign weights to the edges of~$\mathcal{G}_t$.

%%%%%%%%%%%%%%%%%%%%
\subsection{Christoffel words}\label{TDef}
%%%%%%%%%%%%%%%%%%%%

The lower boundary of the snake graphs gives rise to
a {Christoffel word}. This notion plays an important role in the theory of Markov numbers. 
We refer to \cite{Reu} for details.

Fix a rational number $t\in [0,1]$ and consider the associated Cohn's snake graphs defined in the previous section.
Along the lower boundary of  the graph, assign the letter $X$ to each horizontal step of length 1 and the letter $Y$ to each vertical step of length 1.
One can do the same on the lower boundary of $\G_t$ except that in $\G_t$ the initial horizontal step and the final vertical step are of length $\textstyle \half$.
The \textit{Christoffel word} associated to $t$, denoted by $w_t$, is the word in the letters $X, Y$ 
read from left to right along the lower boundary of  the snake graph.

\begin{ex}
From Fig \ref{fig:snakes} one obtains $w_{\frac35}={XXYXXYXY}$.
\end{ex}

Christoffel words label regions of the binary tree according to the following local rule

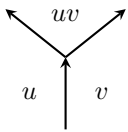
\begin{figure}[H]
\begin{center}
\begin{tikzpicture}[scale=0.8]
% \draw[thick,  ->, >=stealth] (0,0.2) -- (1,1);	
  %\draw[thick,  ->, >=stealth] (1,1)-- (2,0.2);	
    \draw[thick,  ->, >=stealth] (1,1)-- (1,2.2);	
       \draw[thick,  ->, >=stealth] (1,2.2)-- (0,3);	
          \draw[thick,  ->, >=stealth] (1,2.2)-- (2,3);
  %    \node at (1,0.3) {$c$}; 
         \node at (0.4, 1.6) {$u$}; 
          \node at (1.6, 1.6) {$v$}; 
          \node at (1, 2.9) {$uv$}; 
\end{tikzpicture}
%\caption{Fish bone}
\label{fig:fishMarkoff2}
\end{center}
\end{figure}
\noindent
where $u$ and $v$ are words in $X, Y$ and $uv$ is the concatenated word.
The binary tree of Christoffel words  is as follows.

\begin{figure}[H]
\begin{center}
\begin{tikzpicture}[scale=0.7, every node/.style={scale=0.8}]
    \node at (1.75,1.25) {$XY$};
    \node at (-1.75,1.25){$X$};
    \node at (0,4) {$X^2Y$};
    \node at (3.11, 3.8) {$X^2YXY$};
    \draw (2.20, 5.76) node[scale=0.9] {$(X^2Y)^2XY$}; 
    \node at (4.57, 1.98) {$X^2YXYXY$}; 
    \node at (0.7, 6.3) {$\cdots$}; 
   \draw (5, 5.85) node[scale=0.9] {${\small{(X^2Y)^2XYX^2YXY}}$}; 
    \node at (5.4, 3.35) {$\cdots$}; 
    \node at (4.2, 0.5) {$\cdots$}; 
    \node at (-3.11, 3.8) {$X^3Y$};
    \node at (-2.20, 5.76) {$X^3YX^2Y$}; 
    \node at (-4.57, 1.98) {$X^4Y$}; 
    \node at (-0.7, 6.3) {$\cdots$}; 
    \node at (-3.8, 5.85) {$\cdots$}; 
    \draw (-5.4, 3.35) node[scale=0.9]  {$X^4YX^3Y$}; 
    \node at (-4.2, 0.5) {$X^5Y$}; 
    
        \node at (0,-1){$Y$};
            \node at (-2.125,-1.125){$\cdots$};

%\draw[thick, ->] (0, -0.5) -- (0, 0.75);
   % \draw[thick] (0, 0.75) -- (0, 2);
    \draw[thick, ->, >=stealth] (0, 0.5) -- (0, 2);
    
        \draw[thick,  ->, >=stealth] (-1.5, -0.5)-- (0, 0.5);
        
           \draw[thick,  ->, >=stealth] (0, 0.5) -- (1.5, -0.5);
           
          \draw[thick, green] (-1.5,-0.5) -- (-2.75, -0.5);
          
             \draw[thick, green] (-1.5,-0.5) -- (-1.5, -1.75);

    \draw[thick,  ->, >=stealth] (0,2) -- (1.73,3);
        
        \draw[thick,  ->, >=stealth] (1.73, 3) -- (1.99, 4.48);	

            \draw[thick,  ->, >=stealth] (1.99, 4.48) -- (1.31, 5.53);
            
                \draw[thick,  ->, >=stealth] (1.31, 5.53) -- (0.36, 5.87);
                \draw[thick,  ->, >=stealth] (1.31, 5.53) --  (1.23, 6.53);
                
            \draw[thick,  ->, >=stealth] (1.99, 4.48) -- (2.98, 5.24);	
            
                \draw[thick,  ->, >=stealth] (2.98, 5.24) -- (3.24, 6.21);
                \draw[thick,  ->, >=stealth] (2.98, 5.24) -- (3.98, 5.24);

        \draw[thick,  ->, >=stealth] (1.73, 3) -- (3.14, 2.49);

            \draw[thick,  ->, >=stealth] (3.14, 2.49) -- (4.29, 2.97);

                \draw[thick,  ->, >=stealth] (4.29, 2.97) -- (4.79, 3.84);
                \draw[thick,  ->, >=stealth] (4.29, 2.97) -- (5.26, 2.71);
            
            \draw[thick,  ->, >=stealth] (3.14, 2.49) -- (3.72, 1.38);
                
                \draw[thick,  ->, >=stealth] (3.72, 1.38) -- (4.63, 0.96);
                \draw[thick,  ->, >=stealth] (3.72, 1.38) -- (3.55, 0.39);
                
    \draw[thick,  ->, >=stealth] (0,2) -- (-1.73,3);
    
        \draw[thick,  ->, >=stealth] (-1.73, 3) -- (-1.99, 4.48);	
    
            \draw[thick,  ->, >=stealth] (-1.99, 4.48) -- (-1.31, 5.53);
    
                \draw[thick,  ->, >=stealth] (-1.31, 5.53) -- (-0.36, 5.87);
                \draw[thick,  ->, >=stealth] (-1.31, 5.53) -- (-1.23, 6.53);
    
            \draw[thick,  ->, >=stealth] (-1.99, 4.48) -- (-2.98, 5.24);	
    
                \draw[thick,  ->, >=stealth] (-2.98, 5.24) -- (-3.24, 6.21);
                \draw[thick,  ->, >=stealth] (-2.98, 5.24) -- (-3.98, 5.24);
    
        \draw[thick,  ->, >=stealth] (-1.73, 3) -- (-3.14, 2.49);
    
            \draw[thick,  ->, >=stealth] (-3.14, 2.49) -- (-4.29, 2.97);
    
                \draw[thick,  ->, >=stealth] (-4.29, 2.97) -- (-4.79, 3.84);
                \draw[thick,  ->, >=stealth] (-4.29, 2.97) -- (-5.26, 2.71);
    
            \draw[thick,  ->, >=stealth] (-3.14, 2.49) -- (-3.72, 1.38);
    
                \draw[thick,  ->, >=stealth] (-3.72, 1.38) -- (-4.63, 0.96);
                \draw[thick,  ->, >=stealth] (-3.72, 1.38) -- (-3.55, 0.39);
\end{tikzpicture}
\caption{The tree of Christoffel words in $X, Y$}
\label{fig:Christoffel treeXY}
\end{center}
\end{figure}
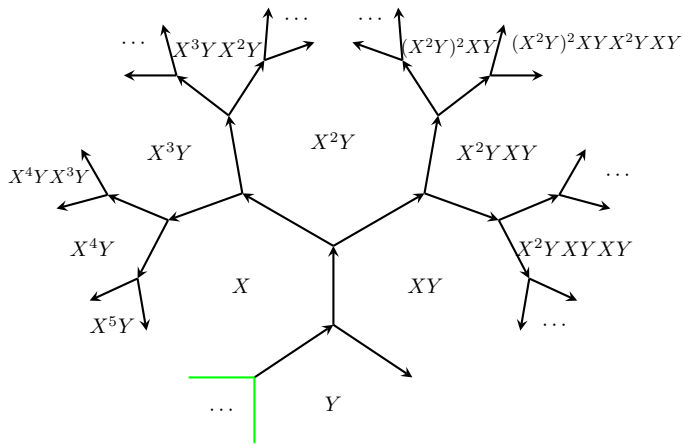

Let us collect here some important facts about Christoffel words.

\begin{enumerate}
\item 
Every Christoffel word is of the form $w=X\pi Y$ where $\pi$ is a palindromic word in $X, Y$;
see \cite[Thm 2.3.1]{Reu}.
\item 
The rational number corresponding to $w_t$ is given by the simple formula:
$$
t=\frac{\# Y\text{'s in } w_t}{\# X\text{'s in } w_t}.
$$
\item Substituting the initial Cohn matrices 
$A(n)$ and $A(n+1)$ in place of $X$ and $Y$ in a Christoffel word $w_t$ lead to the Cohn matrix 
$$
C_t (n)=w_t\left(A(n), A(n+1)\right).
$$

The corresponding Markov number is given by 
$$
m^t=\frac13 \Tr \,M_t(n).
$$
This is a reformulation of Theorem~\ref{CohnThm}.
\end{enumerate}

\begin{rem}
It is often convenient to substitute $A=X$ and $B=XY$ in a Christoffel word $w$. Doing so, one obtains a Christoffel word 
in the letter $A, B$, that we denote by $\tilde w$. For instance, for $t=\frac35$ one has $\tilde w=ABABB$.
In terms of matrices
$$
C_t(n)=w_t\left(A(n), A(n+1)\right)=\tilde w_t \left(A(n), B(n)\right).
$$
\end{rem}
%%%%%%%%%%%%%%%%%
\subsection{Correspondences}
%%%%%%%%%%%%%%%%%

In the previous sections, one has described the following correspondences:
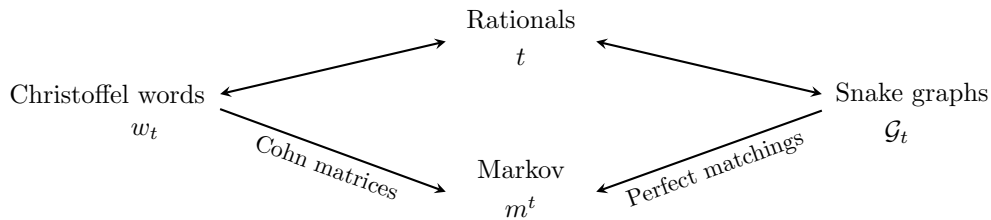
\begin{figure}[H]
\begin{center}
\begin{tikzpicture}
        \node at (0,0){$m^t$};
         \node at (0,0.5){Markov};
        \node at (0,2.5){Rationals};
            \node at (0,2){$t$};
                 \node at (-5,1){$w_t$};
         \node at (-5.5,1.5){Christoffel words};
                  \node at (5,1){$\mathcal{G}_t$};
         \node at (5.2,1.5){Snake graphs};
   \draw[thick,  <-, >=stealth] (1,0.2) --node[below, scale=0.9, rotate=20] {Perfect matchings} (4,1.28);
   %\draw (2.5,0.5) node[scale=0.9, rotate=20]{Perfect matchings};
   
      %\draw (-2.9,0.7) node[scale=0.9, rotate=-20]{Cohn Matrices};
        \draw[thick,  ->, >=stealth] (-4,1.3) --node[below, scale=0.9, rotate=-20] {Cohn matrices} (-1,0.2);
     
        \draw[thick,  <->, >=stealth] (-4,1.5) --(-1,2.2);
         \draw[thick,  <->, >=stealth] (1,2.2) --(4,1.5);
\end{tikzpicture}
%\caption{The tree of Christoffel words in $X, Y$}
%\label{fig:Corres}
\end{center}
\end{figure}
\noindent
We will describe $q$-deformations of Cohn matrices and perfect matchings
that lead to polynomials $m^t_q$.

%%%%%%%%%%%%%%%%%%%%
%%%%%%%%%%%%%%%%%%%%
\section{Proof of the uniqueness result}\label{UniSec}
%%%%%%%%%%%%%%%%%%%%
%%%%%%%%%%%%%%%%%%%%

In this section we prove Theorem~\ref{ExtUniq} stating
that every triple of Markov numbers $(a,b,c)$ has a unique $q$-deformation $(a_q,b_q,c_q)$.

%%%%%%%%%%%%%%%%%%%%
\subsection{Proof of  Proposition~\ref{RecProp}}
%%%%%%%%%%%%%%%%%%%%

We start with the recurrent procedure described in Proposition~\ref{RecProp}.
Recall that $(a_q,b_q,c_q)\in\Z[q^{\pm1}]^3$ 
is called a $q$-Markov triple, if this is a triple of Laurent polynomials with integer coefficients, which are solutions to~\eqref{qMarEq},
and 
$$
(a,b,c):=(a_q,b_q,c_q)\vert_{q=1}
$$ 
is a Markov triple of integers.

\begin{lem}
\label{ConstructLem}
If $(a_q,b_q,c_q)\in\Z[q^{\pm1}]^3$ 
is a $q$-Markov triple, then so is the triple  $(a_q,b_q,c'_q)$ with $c'_q$ given by~\eqref{qMutEq}. 
\end{lem}

\begin{proof}
It is a simple direct computation to check that if
$(a_q,b_q,c_q)$ satisfies~\eqref{qMarEq}, then $(a_q,b_q,c'_q)$ is also a solution to~\eqref{qMarEq}.
As in the classical Markov case, this is just a reformulation of the Vieta formulas for a quadratic equation.
In addition, evaluating at $q=1$, one obtains 
$$
c'=3ab-c=\frac{a^2+b^2}{c} >0,
$$
so that $(a_q,b_q,c'_q)\vert_{q=1}=(a,b,c')$ is a Markov triple.
\end{proof}

We want to show that performing a sequence of transformations of the type~\eqref{qMutEq} on 
an arbitrary $q$-Markov triple 
$(a_q,b_q,c_q)$ one can eventually arrive at the triple $(1,1,1)$.

Denote by  $d_1, d_2, d_3$ the degrees of the polynomials $a_q, b_q, c_q$ respectively,
and by $\a_1, \a_2, \a_3$ the leading coefficients of the polynomials $a_q, b_q, c_q$ respectively.
We assume $d_1\leq d_2 \leq d_3$
and proceed by induction on~$d_3$.

Case $d_1=d_2=d_3=0$. In this case $(a_q,b_q,c_q)=(\a_1, \a_2, \a_3)\in\Z_{>0}^3$  is a Markov triple satisfying in addition the relation
$$
\a_1^2+\a_2^2+\a_3^2=(q+1+q^{-1})\,\a_1 \a_2 \a_3 +2-q-q^{-1}.
$$
This immediately implies $\a_1 \a_2 \a_3 =1$ so that $(\a_1, \a_2, \a_3)=(1, 1,1)$ is the only triple of 
(positive) constant polynomials satisfying~\eqref{qMarEq}.

Case $d_3\geq 1$. Since $(a_q,b_q,c_q)$ satisfies~\eqref{qMarEq}, one obtains the following relations on the degrees and leading coefficients:
$$
2d_3=d_1+d_2+d_3+1, 
\qquad\qquad 
\a_3^2=\a_1\a_2\a_3.
$$
In the expression of $c'_q$ the term  $q^{-1}\left[3\right]_qa_qb_q$ has degree $d_1+d_2+1$ and leading coefficient $\a_1\a_2$. 
The above relations imply that $q^{-1}\left[3\right]_qa_qb_q$ has degree $d_3$ and
the leading coefficient $\a_3$, the same as~$c_q$.
Hence as the higher order term in $c'_q=q^{-1}\left[3\right]_qa_qb_q-c_q$ vanishes,
and therefore $c'_q$ is of degree $d'_3$, which is strictly less than $d_3$.

Repeating successively transformations on the polynomial of the triple with greatest degree  will eventually lead to a triple of polynomials of degree 0,
which as seen above can only be $(1,1,1)$.
Observing that the transformation~\eqref{qMutEq} is an involution one gets Proposition~\ref{RecProp}.

%%%%%%%%%%%%%%%%%%%%
\subsection{End of the proof}
%%%%%%%%%%%%%%%%%%%%

Finally Theorem~\ref{ExtUniq} readily follows from Proposition~\ref{RecProp},
the fact that the transformation~\eqref{qMutEq} is an involution,
together with the statement that a Markov triple $(a,b,c)$ appears only once in the Markov tree.
For the latter; see \cite[Thm 3.3]{Aig}.

%%%%%%%%%%%%%%%%%%%%
%%%%%%%%%%%%%%%%%%%%
\section{$q$-Markov numbers from the modular group}\label{CohnSec}
%%%%%%%%%%%%%%%%%%%%
%%%%%%%%%%%%%%%%%%%%

In this section, we present the $q$-deformed version of the computations
with Cohn matrices, that we explained in Section~\ref{CohnM}.
The $q$-deformed Cohn matrices we consider belong to the $q$-deformed version of $\PSL(2,\Z)$
described in~\cite{LMGadv}.
As explained in~\cite{MGOV}, this $q$-deformation is the projective version of the classical Burau representation.

Note that a specific choice of $q$-deformed Cohn matrices was considered in~\cite{LMGadv,LLS,Kog,Ogu}.
We prove that the resulting $q$-deformed Markov numbers do not depend on the choice of Cohn matrices.

%%%%%%%%%%%%%%%%%%%%
\subsection{$q$-deformed action of $\PSL(2,\Z)$}
%%%%%%%%%%%%%%%%%%%%

Recall that the group $\PSL(2,\Z)$ is generated by any two of the following three elements
\begin{equation}
\label{GenSL}
T=
\begin{pmatrix}
1&1\\[2pt]
0&1
\end{pmatrix},
\qquad
S=
\begin{pmatrix}
0&-1\\[2pt]
1&0
\end{pmatrix},
\qquad
L=TST=
\begin{pmatrix}
1&0\\[2pt]
1&1
\end{pmatrix}.
\end{equation}
All three of these standard generators will be useful for explicit calculations.
The elements $T$ and $S$ satisfy the relations $S^2=(TS)^3=1$.

Following~\cite{MGOfmsigma,LMGadv}, consider the following matrices
\begin{equation}
\label{GenSLq}
T_{q}=
\begin{pmatrix}
q&1\\[2pt]
0&1
\end{pmatrix},
\qquad
S_{q}=
\begin{pmatrix}
0&-1\\[2pt]
q&{\phantom -}0
\end{pmatrix},
\qquad
L_q=T_qS_qT_q=
\begin{pmatrix}
q&0\\[2pt]
q&1
\end{pmatrix}.
\end{equation}
depending on the parameter~$q$.
They correspond to linear-fractional transformations acting on the space~$\Z(q)$
of rational functions in~$q$ with integer coefficients,
that (slightly abusing the notation) we also denote by $T_{q}$ and $S_{q}$:
\begin{equation}
\label{AcTq}
T_q\left(f(q)\right)=qf(q)+1,
\qquad
S_q\left(f(q)\right)=-\frac{1}{qf(q)},
\qquad
L_q\left(f(q)\right)=-\frac{qf(q)}{qf(q)+1}.
\end{equation}
It is very easy to check that
these operators satisfy exactly the same relations as the classical generators of $\PSL(2,\Z)$, namely
$$
S_q^2=(T_qS_q)^3=1.
$$
Therefore, they generate a $\PSL(2,\Z)$-action on $\Z(q)$.

One arrives at the following conclusion that will be important for the sequel.
Consider the action of~$\PSL(2,\Z)$ generated by~\eqref{AcTq}.
Every element of $\PSL(2,\Z)$ corresponds to a linear-fractional transformation on~$\Z(q)$
and can be represented by a $2\times2$ matrix with coefficients that are polynomial in $q$,
defined modulo a scalar multiple, which is a power of~$q$.
To make the choice of a $2\times2$ matrix representing an element of~$\PSL(2,\Z)$ canonical,
it suffices to fix the determinant of the matrix.
In our case, it will always be possible to choose the determinant equal to~$1$.

\begin{rem}
Note also that the matrices
\begin{equation}
\label{GenB}
T_{q}=
\begin{pmatrix}
q&1\\[2pt]
0&1
\end{pmatrix},
\qquad\qquad
L^{-1}=
\begin{pmatrix}
1&0\\[2pt]
-q&q
\end{pmatrix}
\end{equation}
satisfy the braid relation $T_{q}L^{-1}_qT_{q}=L^{-1}_qT_{q}L^{-1}_q$.
This implies that the matrices $T_{q}$ and $L_q$
generate a representation 
$$
\rho_3:B_3\to\mathrm{GL}(2,\Z(q)),
\qquad\qquad
\sigma_1\mapsto T_q,
\quad
\sigma_2\mapsto L^{-1}_q
$$
of the braid group~$B_3$
by $2\times2$ matrices with coefficients that are rational functions in~$q$.
Here $\sigma_1$ and~$\sigma_2$ are the standard generators of~$B_3$.
This representation has been known for almost a hundred years under the name of 
(reduced) Burau representation; see~\cite{Bir}.
\end{rem}

%%%%%%%%%%%%%%%%%%%%
\subsection{$q$-deformed Cohn matrices}\label{qCohnM}
%%%%%%%%%%%%%%%%%%%%

Consider the Cohn matrices $A(n)$ and $B(n)$, as in~\eqref{Cohn}.
Our next goal is to calculate their $q$-deformations $A(n)_q$ and $B(n)_q$.
Since these matrices are defined up to a scalar multiple, as are all of the elements of~$\PSL(2,\Z)$,
it is natural to look for $q$-deformations satisfying the additional condition
\begin{equation}
\label{DetCond}
\det\left(A(n)_q\right)\;=\;
\det\left(B(n)_q\right)\;=\;1.
\end{equation}
This fixes a canonical choice of the matrices representing elements of~$\PSL(2,\Z)$.

We will use the following standard notation for $q$-integers:
$[n]_q:=\dfrac{q^n-1}{q-1}$;
for positive $n$, one has $[n]_q=q^{n-1}+q^{n-2}+\cdots+1$.

We have the following computational statement.
 
\begin{prop}
\label{AdefProp}
The $q$-deformation  of the matrix $A(n)$ satisfying the condition~\eqref{DetCond}
is as follows.
\begin{equation}
\label{AdefEq}
A(n)_q= \begin{pmatrix} 
q^{2-n} [n]_{q} & q^{1-n} \\[6pt]
 [n]_q[3-n]_{q} - q^{n-1} & q^{-1}[3-n]_{q}
\end{pmatrix}.
\end{equation}
\end{prop}

\begin{proof}
First check that $A(n)=L^{3-n}\,S\,L^n$, where $L$ is given by~\eqref{GenSL}.
The $q$-deformation of $A(n)$ is then given by 
$$
A(n)_q:=L_q^{3-n}\,S_q\,L_q^n,
$$
where $L_q$ and $S_q$ are as in~\eqref{GenSLq}.
One has
$$
\begin{array}{rcl}
A(n)_q &=& 
\begin{pmatrix} 
q^{3-n}  & 0 \\[4pt]
q[3-n]_{q}&1
\end{pmatrix}
\begin{pmatrix} 
0 & 1 \\[4pt]
-q&0
\end{pmatrix}
\begin{pmatrix} 
q^{n}  & 0 \\[4pt]
q[n]_{q}&1
\end{pmatrix}\\[18pt]
&=&
\begin{pmatrix} 
q^{4-n} [n]_q & q^{3-n} \\[6pt]
q^2[n]_q[3-n]_q - q^{n+1} & q[3-n]_q 
\end{pmatrix}.
\end{array}
$$
and we divide through by $q^2$ so that the determinant is $1$.
\end{proof}

We will also need the companion Cohn matrix $B(n)_q$, where $B(n)$ is given by~\eqref{Cohn}.

\begin{cor}
The $q$-deformation  of the matrix $B(n)$ satisfying the condition~\eqref{DetCond}
\begin{equation}
\label{BdefEq}
B(n)_q= 
\begin{pmatrix} 
q^{1-n} [n+1]_{q}[2]_{q}-q & q^{-n}[2]_{q} \\[8pt]
q^{-1}[n+1]_{q}[3-n]_{q}[2]_{q}-[n+1]_{q} - q^{n-1}[3-n]_{q} & \; q^{-2}[3-n]_{q}[2]_{q}-q^{-1}
\end{pmatrix}.
\end{equation}
\end{cor}

\begin{proof}
$B(n)$ is given by~\eqref{Cohn} can be expressed in terms of~$A(n)$.
Indeed, $B(n)=A(n)A(n+1),$ and therefore
$$
B(n)_q=A(n)_qA(n+1)_q.
$$
After a simple computation, one deduces~\eqref{BdefEq} from~\eqref{AdefEq}.
\end{proof}

\begin{ex}
\label{UsefulEx}
The following special choices
\begin{eqnarray*}
A(0)_q
=
\begin{pmatrix}
0&q\\[4pt]
-q^{-1}&q^{-1}[3]_q
\end{pmatrix},
&
B(0)_q
=
\begin{pmatrix}
q^2&[2]_q\\[4pt]
q[2]_q& q^{-2}[4]_q+1
\end{pmatrix},
\\[10pt]
A(1)_q
=
\begin{pmatrix}
    q & 1 \\[4pt]
    q & q^{-1}[2]_q
\end{pmatrix},
&
B(1)_q
=
\begin{pmatrix}
    [3]_q & q^{-1}[2]_q \\[4pt]
    q^{-1}[4]_q & q^{-2}[3]_q 
\end{pmatrix},
\\[10pt]
A(2)_q
=
\begin{pmatrix}
    [2]_q& q^{-1} \\[4pt]
  1 & q^{-1}
\end{pmatrix},
&
B(2)_q
=
 \begin{pmatrix}
    q^{-1}[4]_q+1& q^{-2}[2]_q \\[4pt]
    q^{-1}[2]_q & q^{-2}
\end{pmatrix}
\end{eqnarray*}
 are particularly useful in practice
 and will be needed in Section~\ref{SnakeSec}.
 \end{ex}

%%%%%%%%%%%%%%%%%%%%
\subsection{$q$-Markov numbers as traces of Cohn matrices}
%%%%%%%%%%%%%%%%%%%%

Consider now the Cohn tree, as depicted in Fig.~\ref{fig:Christoffel tree AB},
and replace the classical matrices $A(n),B(n)$ by the $q$-deformed matrices $A(n)_q,B(n)_q$.

The main statement of this section is the following.

\begin{thm}
\label{CohnShiftThm}
For every matrix $C_t$ of the Cohn tree, the trace of the $q$-deformed matrix is equal to
$q^{-1}\left[3\right]_qm^t_q$,
where $m^t_q$ is the  $q$-deformed analogue of  $m^t$
defined by \eqref{qMarEq}.
\end{thm}

\begin{proof}
We proceed by induction on the position of the matrices in the Cohn tree.

(1) The induction basis.
From~\eqref{AdefEq}, one has 
\begin{align*}
    \Tr(A(n)_q) &= q^{2-n} [n]_q + q^{-1} [3-n]_q \\
    &= \frac{1}{q-1} \left(q^{2-n} (q^n - 1) + q^{-1} (q^{3-n} - 1)\right) \\
    &= \frac{1}{q-1} \left(q^2 - q^{2-n} + q^{2-n} - q^{-1}\right) \\
    &= \frac{q^3 - 1}{q\left(q-1\right)} \\
    &= q^{-1}[3]_q.
\end{align*}
It is independent of $n$ and is equal to $q^{-1}[3]_q=q+1+q^{-1}$, as required.

A similar computation using~\eqref{BdefEq} gives
$$
 \Tr(B(n)_q) = q^{-1}\left[3\right]_q(q+q^{-1}),
$$
in accordance with the second $q$-Markov number $2_q=q+q^{-1}$.

Note also that 
$$
A(n)_qB^{-1}(n)_q=A(n-2)_q,
$$
and therefore the trace of this matrix is equal to $\Tr(A(n)_q)$:
$$
\Tr\left(A(n)_qB^{-1}(n)_q\right)=q+1+q^{-1}.
$$

(2) The induction step.
Suppose that $q$-deformed Cohn matrices $M$ and $N$ 
are separated by some branch of the Cohn tree 
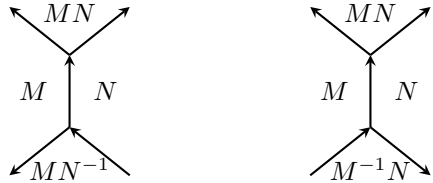
\begin{figure}[H]
\begin{center}
\begin{tikzpicture}[scale=0.8]
 \draw[thick,  <-, >=stealth] (0,0.2) -- (1,1);	
  \draw[thick,  <-, >=stealth] (1,1)-- (2,0.2);	
    \draw[thick,  ->, >=stealth] (1,1)-- (1,2.2);	
       \draw[thick,  ->, >=stealth] (1,2.2)-- (0,3);	
          \draw[thick,  ->, >=stealth] (1,2.2)-- (2,3);
      \node at (1,0.3) {$MN^{-1}$}; 
         \node at (0.4, 1.6) {$M$}; 
          \node at (1.6, 1.6) {$N$}; 
          \node at (1, 2.9) {$MN$}; 
          \begin{scope}[shift={(5,0)}]
 \draw[thick,  ->, >=stealth] (0,0.2) -- (1,1);	
  \draw[thick,  ->, >=stealth] (1,1)-- (2,0.2);	
    \draw[thick,  ->, >=stealth] (1,1)-- (1,2.2);	
       \draw[thick,  ->, >=stealth] (1,2.2)-- (0,3);	
          \draw[thick,  ->, >=stealth] (1,2.2)-- (2,3);
      \node at (1,0.3) {$M^{-1}N$}; 
         \node at (0.4, 1.6) {$M$}; 
          \node at (1.6, 1.6) {$N$}; 
          \node at (1, 2.9) {$MN$}; 
\end{scope}
\end{tikzpicture}
\caption{Cohn branchings.}
\label{fig:fishABC}
\end{center}
\end{figure}

\noindent
and we know that  
\begin{equation}
\label{StepEq}
\Tr(M)=q^{-1}\left[3\right]_qm^t,
\qquad
\Tr(N)=q^{-1}\left[3\right]_qm^{t'},
\qquad
\Tr(MN^{-1})=q^{-1}\left[3\right]_qm^{t\ominus t'},
\end{equation}
where $t\ominus t'$ is the operation opposite to the Farey sum $t\oplus t'$:
$$
\frac{r}{s}\ominus\frac{r'}{s'}:=
\frac{r-r'}{s-s'}.
$$
We need to prove that $\Tr(MN)=q^{-1}\left[3\right]_qm^{t\oplus t'}$.
We give the details in the case Fig.~\ref{fig:fishABC} (left), the second case is similar.

\noindent We use the well-known formula for $2\times2$ matrices
\begin{equation}
\label{TraceEq}
\Tr(MN)=\Tr(M)\Tr(N)-\Tr(MN^{-1}).
\end{equation}
Then~\eqref{StepEq} gives
$$
\Tr(MN)=
q^{-1}\left[3\right]_q\left(q^{-1}\left[3\right]_q m^{t} m^{t'}
-m^{t\ominus t'}\right).
$$
The right-hand side of this equation is nothing but the Vieta transform~\eqref{qMutEq}.

Therefore $\Tr(MN)$ is a multiple of $q^{-1}\left[3\right]_q$, and after division by this Laurent polynomial it satisfies~\eqref{qMarEq}
(cf. Proposition~\ref{RecProp}).
It then follows from Theorem~\ref{ExtUniq} that $\frac{\Tr(MN)}{q+1+q^{-1}}$ is 
a $q$-Markov number corresponding to the Markov number $m^{t\oplus t'}$.
Hence the induction step.

Theorem~\ref{CohnShiftThm} is proved.
\end{proof}

\begin{rem}
Theorem~\ref{CohnShiftThm} means in particular that the traces of Cohn matrices do not depend on the integer parameter~$n$,
and these traces are always a multiple of $q^{-1}\left[3\right]_q$.
Observe also that the upper right elements heavily depend on the choice of initial Cohn matrices.
Note that the definition of $q$-deformed Markov numbers in~\cite{LMGadv,LLS}
is based on a choice of Cohn matrices and makes use of the upper right element of it.
\end{rem}

%%%%%%%%%%%%%%%%%%%%%%%%%
\subsection{Entries of Cohn matrices}
%%%%%%%%%%%%%%%%%%%%%%%%%

It turns out that there are interesting relations between the trace and the entries of the first row and the second column of a
$q$-deformed Cohn matrix,
as suggested in~\cite[Lemma 2.3]{Kog}.

\begin{lem}
 \label{Trace_coeffs_relation}
 Let 
 $$
C=
\begin{pmatrix}
c_{11}&c_{12}\\[4pt]
c_{21}&c_{22}
\end{pmatrix}
$$
be an arbitrary $q$-deformed Cohn matrix.
One then has the following relations
\begin{eqnarray}
    \label{eq1}
\frac{\Tr(C)}{\qth} &=& \left(q^{-1}-q^{-2}\right)c_{11} + q^{-1}c_{12},\\
&=& q^2c_{12}+q\left(1-q\right)c_{22}.
\label{eq2}
\end{eqnarray}

\end{lem}

\begin{proof}
Let us give the details for the relation~\eqref{eq1},
the proof of~\eqref{eq2} is similar.

The proof goes along the same lines as that of Theorem~\ref{CohnShiftThm}.
We proceed by induction on the Cohn tree. 

(1) The induction basis.
A simple direct computation shows that the relation \eqref{eq1} holds for the Cohn matrices 
$A(n)_q$ and $B(n)_q$ in the initial triple.

(2) The induction step.
Let $(M,N,M^{-1}N)$ be a triple of Cohn matrices around a vertex of the tree.
We assume that the relation \eqref{eq1} holds for $C= M, N,$ and $M^{-1}N$,
as in Fig.~\ref{fig:fishABC} (right).
Note that the case Fig.~\ref{fig:fishABC} (left) is similar.
We want to show that the relation holds also for $C=MN$.

Consider the coefficients of the matrices 
$$
M=(m_{ij}),
\quad 
N=(n_{ij}),
\quad 
P=M^{-1}N=(p_{ij}), 
\quad
C=MN=(c_{ij}),
$$ 
and for simplicity set
$$
\a:=q^{-1}-q^{-2}, \qquad \b:=q^{-1}.
$$
The proof of~\eqref{eq1} actually does not depend on the choice of constants $\a$ and~$\b$.

The relation~\eqref{TraceEq} reads
\begin{align*}
\frac{\Tr(MN)}{\qth} &= \qth\left(\frac{\Tr(M)}{\qth} \, \frac{\Tr(N)}{\qth}\right) - \frac{\Tr(P)}{\qth}\\[4pt]
&= ( m_{11} +  m_{22})(\a n_{11} + \b  n_{12}) - (\a p_{11} + \b p_{12}),
\end{align*}

\noindent with  $P = M^{-1}N$ so 
$$
p_{11} = m_{22}n_{11} -m_{12}n_{21}
\qquad\hbox{and}\qquad
p_{12} = m_{22}n_{12} - m_{12}n_{22}.
$$
One then computes for $C=MN$:
\begin{align*}
\frac{\Tr(C)}{\qth} &= (m_{11} + m_{22})(\a n_{11} + \b n_{12}) - (\a (m_{22}n_{11} -m_{12}n_{21}) + \b (m_{22}n_{12} - m_{12}n_{22})),\\
&= \a (m_{11}n_{11} + m_{12}n_{21}) + \b (m_{11}n_{12} + m_{12}n_{22})\\[4pt]
&= \a c_{11} + \b c_{12}.
\end{align*}
Hence the lemma.
\end{proof}

\begin{rem}
It is worth mentionning that the expression~\eqref{eq2} coincides with the numerator of the ``left $q$-deformation'' of the rational~$\frac{c_{12}}{c_{22}}$;
see~\cite{BBL}.
The same expression appears in~\cite{MGOfmsigma}
as the Jones polynomial corresponding to the rational knot parametrized by~$\frac{c_{12}}{c_{22}}$.
\end{rem}

%%%%%%%%%%%%%%%%%%%%
%%%%%%%%%%%%%%%%%%%%
\section{A combinatorial model}\label{SnakeSec}
%%%%%%%%%%%%%%%%%%%%
%%%%%%%%%%%%%%%%%%%%

In this section, we prove Theorem~\ref{SnakeThm}
stating that $q$-Markov numbers count perfect matchings of
snake graphs with weighted edges.

%%%%%%%%%%%%%%%%%%%%
\subsection{Snake graphs with weighted edges}\label{WeghtedSnakeS}
%%%%%%%%%%%%%%%%%%%%
\begin{defn}
The \textit{weighted snake graph} $\mathcal{G}_t(q)$ is obtained from the snake graph 
$\G_t$ by assigning weights on the edges of $\G_t$ according to the following rule.

\begin{enumerate}

\item 
Vertical edges on the Western boundary of $\G_t$ are assigned with weights $q^{-1}, q, \ldots$ alternating from bottom to top.

\item
Horizontal edges on the Southern boundary of $\G_t$ are assigned with weights $q, q^{-1}, \ldots$ alternating from left to right.

\item
Other edges of $\G_t$ are assigned with weight~$1$.

\end{enumerate}

\end{defn}

Alternatively, for $0<t<1$, the graph $\mathcal{G}_t(q)$ can be constructed directly from the Christoffel words.
More precisely we consider the Christoffel word $\tilde w_t (A,B)$ 
obtained from the word $w_t$ by substituting the letters $A=X$, $B=XY$
(see Section~\ref{TDef}).
Each letter in $\tilde w_t $ gives rise to a piece of snake graph according to the following description:

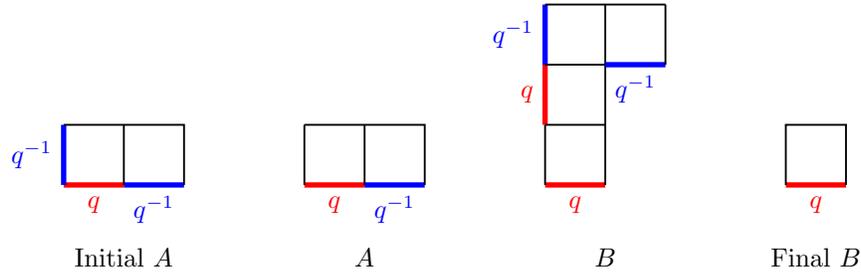
\begin{center}
\begin{tikzpicture}[scale=0.8]
    \begin{scope}[shift={(-4,0)}]
    \draw[line width=2pt,red] (0,0)-- (1,0) node[below,midway] {$q$};
    \draw[line width=2pt,blue] (1,0)-- (2,0) node[below,midway] {$q^{-1}$};
    \draw[line width=2pt,blue] (0,0)-- (0,1) node[left,midway] {$q^{-1}$};
    \draw[line width=0.7pt] (0,1)--(2,1)--(2,0);
    \draw[line width=0.7pt] (1,1)--(1,0);
    \node at(1,-1.2) {Initial $A$};
    \end{scope}

    \begin{scope}
    \draw[line width=2pt,red] (0,0)-- (1,0) node[below,midway] {$q$};
    \draw[line width=2pt,blue] (1,0)-- (2,0) node[below,midway] {$q^{-1}$};
    \draw[line width=0.7pt] (0,0)-- (0,1);
    \draw[line width=0.7pt] (0,1)--(2,1)--(2,0);
    \draw[line width=0.7pt] (1,1)--(1,0);
    \node at(1,-1.2) {$A$};
    \end{scope}

    \begin{scope}[shift={(4,0)}]
    \draw[line width=2pt,red] (0,0)--(1,0) node[below,midway] {$q$};
    \draw[line width=0.7pt] (0,0)--(0,1);
    \draw[line width=2pt,red] (0,1)--(0,2) node[left,midway] {$q$};
    \draw[line width=2pt,blue] (0,2)--(0,3) node[left,midway] {$q^{-1}$};
    \draw[line width=0.7pt] (1,0)--(1,3);
    \draw[line width=0.7pt] (0,3)--(2,3);
    \draw[line width=0.7pt] (0,1)--(1,1);
    \draw[line width=0.7pt] (0,2)--(1,2);
    \draw[line width=2pt,blue] (1,2)--(2,2) node[below,midway] {$q^{-1}$};
    \draw[line width=0.7pt] (2,2)--(2,3);

    \node at(1,-1.2) {$B$};
    \end{scope}

    \begin{scope}[shift={(8,0)}]
    \draw[line width=2pt,red] (0,0)--(1,0) node[below,midway] {$q$};
    \draw[line width=0.7pt] (0,0)--(0,1)--(1,1)--(1,0);
    \node at(0.5,-1.2) {Final $B$};
    \end{scope}
\end{tikzpicture}
\end{center}
All the pieces are connected by gluing the rightmost vertical edge of the previous piece to the bottom leftmost vertical edge of the next piece.

\begin{ex}
(a)
Examples of $\G_t(q)$, for $t=\frac12, \frac13, \frac 23$, corresponding to the Markov numbers $5, 13, 29$, respectively.

\begin{center}
\begin{tikzpicture}[scale=0.8]
    \begin{scope}[shift={(-6,0)}]
    \draw[line width=2pt,red] (0,0)-- (1,0) node[below,midway] {$q$};
    \draw[line width=2pt,blue] (1,0)-- (2,0) node[below,midway] {$q^{-1}$};
    \draw[line width=2pt,blue] (0,0)-- (0,1) node[left,midway] {$q^{-1}$};
    \draw[line width=0.7pt] (0,1)--(2,1)--(2,0);
    \draw[line width=0.7pt] (1,1)--(1,0);
        \draw[line width=2pt,red] (2,0)--(3,0) node[below,midway] {$q$};
    \draw[line width=0.7pt] (3,0)--(3,1)--(2,1);
    \node at(1,-1.2) {$\tilde w_t=AB$};
 \draw (0.5, 0.5) node[] {$2$};
		\draw (1.5, 0.5) node[] {$3$};
		\draw (2.5, 0.5) node[] {$5$};
    \end{scope}
    
        \begin{scope}[shift={(-1,0)}]
    \draw[line width=2pt,red] (0,0)-- (1,0) node[below,midway] {$q$};
    \draw[line width=2pt,blue] (1,0)-- (2,0) node[below,midway] {$q^{-1}$};
    \draw[line width=0.7pt] (0,0)-- (0,1) node[left,midway] {$q^{-1}$};
    \draw[line width=0.7pt] (0,1)--(2,1)--(2,0);
    \draw[line width=0.7pt] (1,1)--(1,0);
        \draw[line width=2pt,red] (2,0)--(3,0) node[below,midway] {$q$};
    \draw[line width=0.7pt] (3,0)--(3,1)--(2,1);
            \draw[line width=2pt,blue] (3,0)--(4,0) node[below,midway] {$q^{-1}$};
    \draw[line width=0.7pt] (4,0)--(4,1)--(3,1);
    \node at(1,-1.2) {$\tilde w_t=A^2B$};
 %   \node at(1,2) { $t=\frac12$};
    \end{scope}
       \begin{scope}[shift={(3,0)}]
    \draw[line width=2pt,red] (0,0)--(1,0) node[below,midway] {$q$};
    \draw[line width=0.7pt] (0,0)--(0,1)--(1,1)--(1,0);
   \draw (-3.5, 0.5) node[] {$2$};
		\draw (-2.5, 0.5) node[] {$3$};
		\draw (-1.5, 0.5) node[] {$5$};
		\draw (-0.5, 0.5) node[] {$8$};
		\draw (0.5, 0.5) node[] {$13$};
    \end{scope}

    \begin{scope}[shift={(6,0)}]
    \draw[line width=2pt,red] (0,0)-- (1,0) node[below,midway] {$q$};
    \draw[line width=2pt,blue] (1,0)-- (2,0) node[below,midway] {$q^{-1}$};
    \draw[line width=2pt,blue] (0,0)-- (0,1) node[left,midway] {$q^{-1}$};
    \draw[line width=0.7pt] (0,1)--(2,1)--(2,0);
    \draw[line width=0.7pt] (1,1)--(1,0);
    \node at(1,-1.2) {$\tilde w_t=AB^2$};
    		\draw (0.5, 0.5) node[] {$2$};
		\draw (1.5, 0.5) node[] {$3$};
		\draw (2.5, 0.5) node[] {$5$};
		\draw (2.5, 1.5) node[] {$7$};
		\draw (2.5, 2.5) node[] {$12$};
		\draw (3.5, 2.5) node[] {$17$};
		\draw (4.5, 2.5) node[] {$29$};
    \end{scope}

    \begin{scope}[shift={(8,0)}]
    \draw[line width=2pt,red] (0,0)--(1,0) node[below,midway] {$q$};
    \draw[line width=0.7pt] (0,0)--(0,1);
    \draw[line width=2pt,red] (0,1)--(0,2) node[left,midway] {$q$};
    \draw[line width=2pt,blue] (0,2)--(0,3) node[left,midway] {$q^{-1}$};
    \draw[line width=0.7pt] (1,0)--(1,3);
    \draw[line width=0.7pt] (0,3)--(2,3);
    \draw[line width=0.7pt] (0,1)--(1,1);
    \draw[line width=0.7pt] (0,2)--(1,2);
    \draw[line width=2pt,blue] (1,2)--(2,2) node[below,midway] {$q^{-1}$};
    \draw[line width=0.7pt] (2,2)--(2,3);
   % \node at(1,-1.2) {$B$};
    \end{scope}
     \begin{scope}[shift={(10,2)}]
    \draw[line width=2pt,red] (0,0)--(1,0) node[below,midway] {$q$};
    \draw[line width=0.7pt] (0,0)--(0,1)--(1,1)--(1,0);
   % \node at(0.5,-1.2) {Final $B$};
    \end{scope}

   \end{tikzpicture}
\end{center}

(b)
Figure~\ref{SnakeFig1} gives an example for $t=\frac35$.
\end{ex}

%%%%%%%%%%%%%%%%%%%%%%%%%
\subsection{Detailed statement of the main theorem}
%%%%%%%%%%%%%%%%%%%%%%%%%

Let us first recall the standard definitions.

\begin{defn}
(i)
 Given a graph $\G$ with weighted edges,
 we  will use the standard notion of the {\it weighted perfect matching} in $\G$.
For a given perfect matching $\mathfrak{m}$, we define its weight, ${\rm{wt}}(\mathfrak{m})$, as the product of the weights of its edges.
The weight is a power of~$q$.

(ii)
Given a snake graph $\mathcal{G}_t(q)$ with weighted edges, 
we define the {\it the weighted number of perfect matchings}
as the generating function for the weighted perfect matchings:
$$
\mu^t(q):=
\sum_{\mathfrak{m}}
{\rm{wt}}(\mathfrak{m}).
$$
Other terms sometimes used in the dimer theory are ``statistics'' or ``statistical sum''. 
\end{defn}

The following statement is a more detailed reformulation of our main result
announced in the introduction as Theorem~\ref{SnakeThm}.

\begin{thm}
\label{SnakeThmBis}
The $q$-deformed Markov number $m_q^t$ is
equal to the weighted number of perfect matchings $\mu^t(q)$ in the snake graph $\mathcal{G}_t(q)$. 
\end{thm}

Our proof, given in the next section, will follow that of Theorem 7.12 of~\cite{Aig}.

%%%%%%%%%%%%%%%%%%%%%%%%%
\subsection{Proof of the main theorem}
%%%%%%%%%%%%%%%%%%%%%%%%%

Fix a rational $t \in\ ]0,1[$ and the corresponding Christoffel word 
$$
\tilde w_t = Aw_1\cdots w_sB,
$$ 
with letters $w_i \in \{A, B\}$.
Let $n$ be the number of boxes in the snake graph $\mathcal{G}_t(q)$. 
We count the boxes travelling along the snake from left to right and bottom to top.

For each $i\in \{1,\cdots,n\}$, let us denote by $\mu_i$ 
the weighted number of perfect matchings of the partial snake graph containing the first $i$ boxes. 
It is convenient to label the $i$-th box with the value $\mu_i$.
In particular, for the initial snake graph one has 

\begin{center}
\begin{tikzpicture}[scale=0.8]
    \begin{scope}
    \draw[line width=2pt,red] (0,0)-- (1,0) node[below,midway] {$q$};
    \draw[line width=2pt,blue] (1,0)-- (2,0) node[below,midway] {$q^{-1}$};
    \draw[line width=2pt,blue] (0,0)-- (0,1) node[left,midway] {$q^{-1}$};
    \draw[line width=0.7pt] (0,1)--(2,1)--(2,0);
    \draw[line width=0.7pt] (1,1)--(1,0);
    \node at(0.5,0.5) {$\mu_1$};
      \node at(1.5,0.5) {$\mu_2$};
    \end{scope}
\end{tikzpicture}
\end{center}

\noindent 
with 
\begin{equation}
\label{Mu1Mu2}
\mu_1 = q+q^{-1},
\qquad\qquad
\mu_2= q+ q^{-1}+q^{-2}.
\end{equation}
Here
$\mu_1$ is easily obtained by enumerating the 2 perfect matchings on a single box and 
$\mu_2$ is obtained by enumerating the 3 perfect matchings on the graph with two boxes
(see example of Fig.~\ref{PerFig}). 

We can compute recursively $\mu_i$ in the snake graph $\mathcal{G}_t(q)$, 
looking at how it evolves through the addition of a type $A$ piece or a type $B$ piece of graph (
see Section~\ref{WeghtedSnakeS} and~\cite{Aig}, p.146). \\

\noindent \underline{Attachment of an $A$-piece} 
\begin{center}
\begin{tikzpicture}[scale=0.8]
    \draw[line width=2pt,red] (0,0)-- (1,0) node[below,midway] {$q$};
    \draw[line width=2pt,blue] (1,0)-- (1.95,0) node[below,midway] {$q^{-1}$};
    \draw[line width=0.7pt] (0,0)-- (0,1);
    \draw[line width=0.7pt] (0,1)--(1.95,1);
      \draw[line width=0.7pt] (1.95,1)--(1.95,0);
    \draw[line width=0.7pt] (1,1)--(1,0);
    \node at(0.5,0.5) {$\a_1$};
    \node at(1.5,0.5) {$\a_2$};

    \draw[line width=2pt,red] (2,0)-- (3,0) node[below,midway] {$q$};
    \draw[line width=2pt,blue] (3,0)-- (4,0) node[below,midway] {$q^{-1}$};
    \draw[line width=0.7pt] (2,0)-- (2,1);
    \draw[line width=0.7pt] (2,1)--(4,1)--(4,0);
    \draw[line width=0.7pt] (3,1)--(3,0);
    \node at(2.5,0.5) {$\b_1$};
    \node at(3.5,0.5) {$\b_2$};
\end{tikzpicture}
\end{center}

\noindent 
Then 
$$
\b_1 = \a_2 + q\a_1,
\qquad\qquad
\b_2 = \b_1 + q^{-1}\a_2,
$$
which can be rewritten as 
$$
\begin{pmatrix}
   \b_1\\[2pt]
   \b_2\\
\end{pmatrix}
=
\begin{pmatrix}
    q & 1\\[2pt]
    q & 1+q^{-1}\\
\end{pmatrix}
\begin{pmatrix}
   \a_1\\[2pt]
   \a_2\\
\end{pmatrix}.
$$
Observe that the matrix 
$\begin{pmatrix}
    q & 1\\[2pt]
    q & 1+q^{-1}\\
\end{pmatrix}$
is nothing but the $q$-deformed Cohn matrix $A(1)_q$, already calculated in Section~\ref{qCohnM}.\\

\noindent \underline{Attachment of a $B$-piece (except for the final one)}

\begin{center}
\begin{tikzpicture}[scale=0.8]
     \draw[line width=2pt,red] (0,0)-- (1,0) node[below,midway] {$q$};
    \draw[line width=2pt,blue] (1,0)-- (1.95,0) node[below,midway] {$q^{-1}$};
    \draw[line width=0.7pt] (0,0)-- (0,1);
    \draw[line width=0.7pt] (0,1)--(1.95,1);
      \draw[line width=0.7pt] (1.95,1)--(1.95,0);
    \draw[line width=0.7pt] (1,1)--(1,0);
    \node at(0.5,0.5) {$\a_1$};
    \node at(1.5,0.5) {$\a_2$};

    \draw[line width=2pt,red] (2,0)--(3,0) node[below,midway] {$q$};
    \draw[line width=0.7pt] (2,0)--(2,1);
    \draw[line width=2pt,red] (2,1)--(2,2) node[left,midway] {$q$};
    \draw[line width=2pt,blue] (2,2)--(2,3) node[left,midway] {$q^{-1}$};
    \draw[line width=0.7pt] (3,0)--(3,3);
    \draw[line width=0.7pt] (2,3)--(4,3);
    \draw[line width=0.7pt] (2,1)--(3,1);
    \draw[line width=0.7pt] (2,2)--(3,2);
    \draw[line width=2pt,blue] (3,2)--(4,2) node[below,midway] {$q^{-1}$};
    \draw[line width=0.7pt] (4,2)--(4,3);
    \node at(2.5,2.5) {$\b_1$};
    \node at(3.5,2.5) {$\b_2$};
    
        \node at(2.5,1.5) {$\g_2$};
    \node at(2.5,0.5) {$\g_1$};
\end{tikzpicture}
\end{center}
One easily obtains the following relations between the weighted number of perfect matchings:
\begin{eqnarray*}
\g_1&=& \a_2+q\a_1\\[4pt]
\g_2&=& \g_1+q^2\a_1\\[4pt]
\b_1&=& \g_2+q^{-1}\g_1\\[4pt]
\b_2&=& \b_1+q^{-2}\g_1
\end{eqnarray*}
from which one deduces
\begin{eqnarray*}
\b_1&=& \left(q^2+q+1\right)\a_1+\left(1+q^{-1}\right)\a_2\\[4pt]
\b_2&=& \left(q^2+q+1+q^{-1}\right)\a_1+ \left(1+q^{-1}+q^{-2}\right)\a_2
\end{eqnarray*}
which again can be rewritten as 
$$
\begin{pmatrix}
   \b_1\\
   \b_2\\
\end{pmatrix}
=
B(1)_q
\begin{pmatrix}
   \a_1\\
   \a_2\\
\end{pmatrix},
$$
where $B(1)_q$ is the $q$-deformed Cohn Matrix companion to $A(1)_q$.\\

\noindent \underline{Attachment of the final $B$-piece}
\begin{center}
\begin{tikzpicture}[scale=0.8]
    \draw[line width=2pt,red] (0,0)-- (1,0) node[below,midway] {$q$};
    \draw[line width=2pt,blue] (1,0)-- (1.95,0) node[below,midway] {$q^{-1}$};
    \draw[line width=0.7pt] (0,0)-- (0,1);
    \draw[line width=0.7pt] (0,1)--(1.95,1);
      \draw[line width=0.7pt] (1.95,1)--(1.95,0);
    \draw[line width=0.7pt] (1,1)--(1,0);
    \node at(0.5,0.5) {$\a_1$};
    \node at(1.5,0.5) {$\a_2$};

    \draw[line width=2pt,red] (2,0)--(3,0) node[below,midway] {$q$};
    \draw[line width=0.7pt] (2,0)--(2,1);
    \draw[line width=0.7pt] (2,1)--(3,1);
    \draw[line width=0.7pt] (3,1)--(3,0);
    \node at(2.5,0.5) {$\mu_n$};
\end{tikzpicture}
\end{center}

\noindent 
In this case, we have
$$
\mu_n = 
\begin{pmatrix} 
q &1
\end{pmatrix} 
\begin{pmatrix}
  \a_1\\[2pt]
   \a_2\\
\end{pmatrix}.
$$
We conclude that the weighted number of perfect matchings of $\mathcal{G}_t(q)$ is given by 
$$
\mu_n=
\begin{pmatrix} 
q &1
\end{pmatrix} 
M_s\cdots M_1
\begin{pmatrix}
  \mu_1\\[2pt]
   \mu_2\\
\end{pmatrix},
$$
where for $k\in \{1,\cdots,s\}$, 
$$
M_k = \begin{cases}
A(1)_q, \; \text{ if } w_k = A,\\[4pt]
B(1)_q, \; \text{ if } w_k = B.\\
\end{cases}
$$
and where $\mu_1$ and $\mu_2$ are given by~\eqref{Mu1Mu2}.

Observing now that 
$$
\begin{pmatrix}
  \mu_1\\[2pt]
   \mu_2\\
\end{pmatrix}
=
B(1)_q
\begin{pmatrix}
q^{-1}-q^{-2}\\[4pt]
q^{-1}\\
\end{pmatrix},
\qquad\hbox{and}\qquad
\begin{pmatrix} 
q &1
\end{pmatrix} 
=
\begin{pmatrix} 
1 &0
\end{pmatrix} 
A(1)_q,
$$
and using the fact that $w_1\cdots w_s$ is a palindrome, we get
$$
\mu_n=
\begin{pmatrix} 
1 &0
\end{pmatrix} 
A(1)_q\,M_1\cdots M_s\,B(1)_q
\begin{pmatrix}
 q^{-1}-q^{-2}\\[4pt]
q^{-1}\\
\end{pmatrix}.
$$
The matrix 
$$
C_t = A(1)_qM_1\cdots M_s B(1)_q=:(m_{ij})
$$
is precisely the Cohn $q$-matrix corresponding to the Markov number $m^t_q$, therefore
$$
\mu_n  = 
\begin{pmatrix} 
1 &0
\end{pmatrix} 
M_t
\begin{pmatrix}
 q^{-1}-q^{-2}\\[4pt]
q^{-1}\\
\end{pmatrix}
=
\begin{pmatrix}
q^{-1}-q^{-2} & q^{-1}
\end{pmatrix} 
\begin{pmatrix}
m_{11} \\[4pt]
m_{12}\\
\end{pmatrix}= \left(q^{-1}-q^{-2}\right)m_{11} + q^{-1}m_{12}.
$$
Applying Lemma \ref{Trace_coeffs_relation}, we finally get
$$
\mu^t(q)=\mu_n =\frac{\Tr(C_t)}{\qth} = m^t_q,
$$
as desired.

%%%%%%%%%%%%%%%%%%%%
\subsection{Properties of the $q$-Markov polynomials}\label{ShapeS}
%%%%%%%%%%%%%%%%%%%%

In this section, we prove Corollary~\ref{PropertThm} and Proposition~\ref{PalindProp}.

Part~(i).
Given a $q$-Markov number $m^t_q$, that is, a Laurent polynomial in~$q$.
It follows from Theorem~\ref{SnakeThmBis}, that
every coefficient of this polynomial counts perfect matchings with some fixed weight.
Therefore this coefficient is non-negative.

Part~(ii).
For every snake graph $\G_t$ corresponding to a Markov number $m^t$,
there exists a unique perfect matching of maximal (or minimal) weight.
Indeed, choosing all edges of weight~$q$ (red edges),
it is easy to see that
there exists a unique way to complete them as a perfect matching.
Therefore, the Laurent polynomial $m^t_q$ has highest (and lowest) order coefficient~$1$.

Part~(iii).
The palindromic property $m^t_{q^{-1}}=m^t_q$ follows directly from~\eqref{qMutEq}
and Proposition~\ref{RecProp}.
Indeed, transformations~\eqref{qMutEq} do not affect the palindromic property, which obviously 
holds for the initial triple.

Proposition~\ref{PalindProp}.
It was proved in~\cite[Thm. 7.3]{Ogu} that $\Tr(C_t)/[3]_q$ is given by the rank polynomial
of a circular fence poset, and these polynomials were proved to be unimodal  in~\cite[Corollary 7.4]{OgRa},
except in some exceptional case.
In the case of Cohn matrices the only possible exception corresponds to the matrix $C_{\frac11}$, i.e., the $q$-Markov number $2_q$.

%%%%%%%%%%%%%%%%%%%%
\subsection{The weighted uniqueness conjecture}\label{InjSec}
%%%%%%%%%%%%%%%%%%%%

Let us now prove Corollary~\ref{UniConj}.
The proof consists of calculating an explicit formula to recover the rational number corresponding to
a $q$-Markov number.

\begin{lem}
\label{RatLem}
Let $m_q^t=q^d+\a q^{d-1}+\cdots+\a q^{1-d}+q^{-d}$ be a $q$-Markov number.
The corresponding rational label is expressed in terms of
the degree of $m_q^t$ and  its coefficient of the term of degree $d-1$:
$$
t=\frac{d-\a}{\a+1}.
$$
\end{lem}

\begin{proof}
There exists exactly one perfect matching of $\G_t$ with weight $q^{-d}$, 
namely ``the minimal matching''  consisting of all boundary edges of weight $q^{-1}$ 
completed (in a unique way) with boundary edges of weight~$1$. 
One has:
$$
d={\# \text{boxes in Cohn's snake}},
$$
that is, the number of letters in the Christoffel word decreased by $1$, i.e.  $d= ({\# X} + {\# Y\text{'s in } w_t})-1$.

We know, by palindromicity, that $\a$ is also the coefficient of the term of degree $1-d$, 
hence the number of perfect matching of $\G_t$ of weight $q^{1-d}$. 
Such a perfect matching is obtained by  taking all edges of weight $q^{-1}$, but one. 
This is possible only by removing in the minimal matching a horizontal edge of weight $q^{-1}$ with the opposite boundary edge sharing the same box, and replacing these two by the two internal vertical edges of the box. There are $(\# X\text{'s}-1)$ horizontal edges of weight $q^{-1}$, therefore $\a=({\# X\text{'s in } w_t})-1$. 

Finally one gets
$$
t=\frac{\# Y\text{'s in } w_t}{\# X\text{'s in } w_t} = \frac{d-\a}{\a+1}.
$$
Hence the lemma.
\end{proof}

Corollary~\ref{UniConj} follows.

%%%%%%%%%%%%%%%%%%%%
%%%%%%%%%%%%%%%%%%%%
\section{Further directions}\label{ExtrSec}
%%%%%%%%%%%%%%%%%%%%
%%%%%%%%%%%%%%%%%%%%

We finish this article with miscellaneous observations that did not fit into the main core
but constitute subjects for future research.

%%%%%%%%%%%%%%%%%%%%
\subsection{Through the looking glass: sixing of the tree~\label{Mirror}}
%%%%%%%%%%%%%%%%%%%%

Our first observation concerns the classical Markov numbers and Cohn matrices.

The usual way to represent the Markov numbers in the form of a tree
uses a binary tree growing in one direction; see Fig.~\ref{fig:Markov Conway}.
This tree starts from the Markov triple $(1, 2, 5)$.
However, the Markov numbers tree can be extended infinitely in all directions, 
 which leads to what is effectively $6$ copies of the Markov numbers tree, 
 starting instead from the initial triple $(1, 1, 1)$.
 We find this ``sixing'' of the tree quite relevant,
 because it better illustrates the role of Cohn matrices and extends Aigner's classification of them~\cite{Aig}.
 
\begin{figure}[H]
\begin{center}
\begin{tikzpicture}[scale=0.75, every node/.style={scale=0.75}]
    \node at (1.75,0.75) {$1$};
    \node at (-1.75,0.75){$1$};
    \node at (0,4) {$1$};
    \node at (3.11, 3.8) {$2$};
    \node at (2.20, 5.76) {$5$}; 
    \node at (4.37, 1.98) {$5$}; 
    \node at (-3.11, 3.8) {$2$};
    \node at (-2.20, 5.76) {$5$}; 
    \node at (-4.37, 1.98) {$5$}; 

    \node at (0, -1.5) {$2$};
    \node at (3.11, -1.8) {$5$};
    \node at (-3.11, -1.8) {$5$};

%\draw[thick, ->] (0, -0.5) -- (0, 0.75);
   % \draw[thick] (0, 0.75) -- (0, 2);
    \draw[thick] (0, 0) -- (0, 2);

    \draw[thick] (0,2) -- (1.73,3);
        
        \draw[thick] (1.73, 3) -- (1.99, 4.48);	

            \draw[thick] (1.99, 4.48) -- (1.31, 5.53);                
            \draw[thick] (1.99, 4.48) -- (2.98, 5.24);

        \draw[thick] (1.73, 3) -- (3.14, 2.49);

            \draw[thick] (3.14, 2.49) -- (4.29, 2.97);            
            \draw[thick] (3.14, 2.49) -- (3.72, 1.38);
                
    \draw[thick] (0,2) -- (-1.73,3);
    
        \draw[thick] (-1.73, 3) -- (-1.99, 4.48);	
    
            \draw[thick] (-1.99, 4.48) -- (-1.31, 5.53);    
            \draw[thick] (-1.99, 4.48) -- (-2.98, 5.24);	
    
        \draw[thick] (-1.73, 3) -- (-3.14, 2.49);
    
            \draw[thick] (-3.14, 2.49) -- (-4.29, 2.97);   
            \draw[thick] (-3.14, 2.49) -- (-3.72, 1.38);

    \draw[thick] (0, 0) -- (1.73, -1);
        \draw[thick] (1.73, -1) -- (1.99, -2.48);
        \draw[thick] (1.73, -1) -- (3.14, -0.49);

    \draw[thick] (0, 0) -- (-1.73, -1);
        \draw[thick] (-1.73, -1) -- (-1.99, -2.48);
        \draw[thick] (-1.73, -1) -- (-3.14, -0.49);
    
\end{tikzpicture}
\caption{Markov numbers on the Conway topograph, in all directions.}
\label{fig:Markov Conway ALL}
\end{center}
\end{figure}
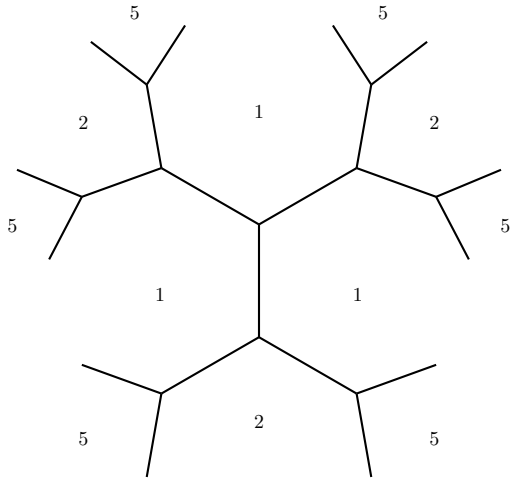

In order to replicate this in the Cohn matrix case, we need to ``glue'' together copies of the tree.
The result of this procedure leads to the following ``complete Cohn tree'',
where $A(n)$ and $B(n)$ are as in~\eqref{Cohn}.

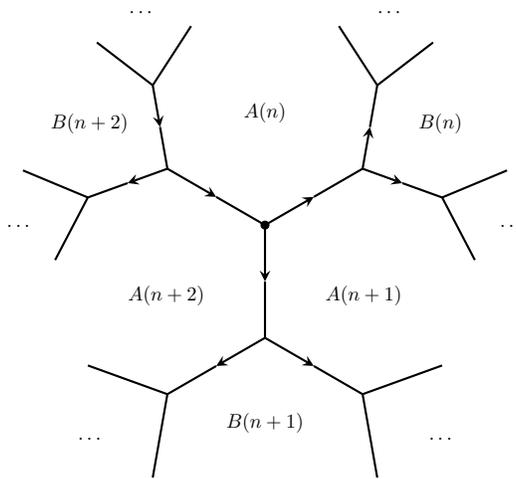
\begin{figure}[H]
\begin{center}
\begin{tikzpicture}[scale=0.75, every node/.style={scale=0.75}]
    \node at (1.75,0.75) {$A(n+1)$};
    \node at (-1.75,0.75){$A(n+2)$};
    \node at (0,4) {$A(n)$};
    \node at (3.11, 3.8) {$B(n)$};
    \node at (2.20, 5.76) {$\cdots$}; 
    \node at (4.37, 1.98) {$\cdots$}; 
    \node at (-3.11, 3.8) {$B(n+2)$};
    \node at (-2.20, 5.76) {$\cdots$}; 
    \node at (-4.37, 1.98) {$\cdots$}; 

    \node at (0, -1.5) {$B(n+1)$};
    \node at (3.11, -1.8) {$\cdots$};
    \node at (-3.11, -1.8) {$\cdots$};

    \filldraw (0, 2) circle (2pt);

%\draw[thick, ->] (0, -0.5) -- (0, 0.75);
   % \draw[thick] (0, 0.75) -- (0, 2);
    \draw[thick] (0, 0) -- (0, 1);
    \draw[thick,  <-, >=stealth] (0, 1) -- (0, 2);

    \draw[thick,  ->, >=stealth] (0,2) -- (0.865,2.5);
    \draw[thick] (0.865,2.5) -- (1.73,3);
        
        \draw[thick,  ->, >=stealth] (1.73, 3) -- (1.86, 3.74);
        \draw[thick,] (1.86, 3.74) -- (1.99, 4.48);	

            \draw[thick,  ] (1.99, 4.48) -- (1.31, 5.53);                
            \draw[thick,] (1.99, 4.48) -- (2.98, 5.24);

        \draw[thick,  ->, >=stealth] (1.73, 3) -- (2.435, 2.745);
        \draw[thick,] (2.435, 2.745) -- (3.14, 2.49);

            \draw[thick,] (3.14, 2.49) -- (4.29, 2.97);            
            \draw[thick,] (3.14, 2.49) -- (3.72, 1.38);
                
    \draw[thick] (0,2) -- (-0.865,2.5);
    \draw[thick,  <-, >=stealth] (-0.865,2.5) -- (-1.73,3);
    
        \draw[thick, ] (-1.73, 3) -- (-1.86, 3.74);
        \draw[thick,  <-, >=stealth] (-1.86, 3.74) -- (-1.99, 4.48);
    
            \draw[thick, ] (-1.99, 4.48) -- (-1.31, 5.53);    
            \draw[thick, ] (-1.99, 4.48) -- (-2.98, 5.24);
    
        \draw[thick,  ->, >=stealth] (-1.73, 3) -- (-2.435, 2.745);
        \draw[thick, ] (-2.435, 2.745) -- (-3.14, 2.49);
    
            \draw[thick, ] (-3.14, 2.49) -- (-4.29, 2.97);   
            \draw[thick, ] (-3.14, 2.49) -- (-3.72, 1.38);

    \draw[thick,  ->, >=stealth] (0, 0) -- (0.865, -0.5);
    \draw[thick, ] (0.865, -0.5) -- (1.73, -1);
        \draw[thick, ] (1.73, -1) -- (1.99, -2.48);
        \draw[thick, ] (1.73, -1) -- (3.14, -0.49);

    \draw[thick,  ->, >=stealth] (0, 0) -- (-0.865, -0.5);
    \draw[thick, ] (-0.865, -0.5) -- (-1.73, -1);
        \draw[thick, ] (-1.73, -1) -- (-1.99, -2.48);
        \draw[thick, ] (-1.73, -1) -- (-3.14, -0.49);
    
\end{tikzpicture}
\caption{Cohn tree in all directions.}
\label{fig:Cohn Conway ALL General}
\end{center}
\end{figure}

\noindent
The tree is oriented, at every vertex of the tree there are exactly {\it one ingoing and two outgoing} edges.
Without this condition, the construction is not consistent.
The rule for constructing the tree and calculating the matrices in it is as in the usual case:
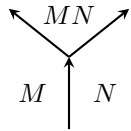
\begin{figure}[H]
\begin{center}
\begin{tikzpicture}[scale=0.8]
    \draw[thick,  ->, >=stealth] (1,1)-- (1,2.2);	
       \draw[thick,  ->, >=stealth] (1,2.2)-- (0,3);	
          \draw[thick,  ->, >=stealth] (1,2.2)-- (2,3);
         \node at (0.4, 1.6) {$M$}; 
          \node at (1.6, 1.6) {$N$}; 
          \node at (1, 2.9) {$MN$}; 
\end{tikzpicture}
\end{center}
\end{figure}

One can prove the following.
\begin{enumerate}

\item
The initial triple of Cohn matrices is $A(n),~A(n+1),~A(n+2)$ is the only possible, so that the tree at Fig.~\ref{fig:Cohn Conway ALL General}
is the most general one.

\item
This complete Cohn tree is consistent for any choice of orientation.

\end{enumerate}
The proof goes along the lines of the proof of Aigner's~\cite{Aig} Theorem 4.7, with only a minor modification,
and we do not go into the details.

We observe from the tree at Fig.~\ref{fig:Cohn Conway ALL General} that besides the thoroughly studied pair of Cohn matrices
$(A(n),B(n))$, several other pairs naturally appear.
For instance, the ``inverted pair'' $(B(n),A(n+1))$ which is not a part of Aigner's classification
and some other interesting cases, according to the choice of orientation.

%%%%%%%%%%%%%%%%%%%%
\subsection{Combinatorial models for $q$-rationals}
%%%%%%%%%%%%%%%%%%%%

In a Cohn matrix the Markov number appears twice, as the top right entry and as the third of the trace.
Therefore with the $q$-deformed Cohn matrix one obtains two $q$-deformations of the Markov number.
As we have seen in this article the $q$-deformation given by the trace is independent of the choice of the initial Cohn matrix.
This is not the case of the top right entry as one can immediately see in the examples of matrices $A(n)_q$ and $B(n)_q$
in Section~\ref{qCohnM} that the entries depend on $n$.

Fixing $n=1$, we will denote by $\tilde m ^t_q$ the top right entry of the Cohn matrix $M_t$ obtained from the initial $A(1)_q$ and $B(1)_q$.
The proof of Theorem~\ref{SnakeThmBis} can be easily adapted in order to obtain a combinatorial model for $\tilde m ^t_q$.

More precisely,  defining $\tilde\G_t(q)$ as the weighted graph obtained from $\G_t(q)$ by removing the weight $q$ of the first horizontal edge, one obtains:
\begin{prop}
The $q$-deformed Markov number $\tilde m ^t_q$ is the weighted number of perfect matchings of $\tilde\G_t(q)$.
\end{prop}
\begin{ex}
For $t=\frac23$, the top right entry of the matrix  $A(1)_qB(1)_q^2$ is 
$$\tilde m ^t_q=q^3 + 3q^2 + 5q + 6+6q^{-1} + 5q^{-2} + 2q^{-3} + q^{-4} $$
and coincides with the weighted number of perfect matchings of the graph below.
\begin{figure}[H]
\begin{center}
\begin{tikzpicture}[scale=0.85, every node/.style={scale=0.75}]
    \begin{scope}[shift={(-2,0)}]
%    \draw[line width=2pt,red] (0,0)-- (1,0) node[below,midway] {$q$};
    \draw[line width=2pt,blue] (1,0)-- (2,0) node[below,midway] {$q^{-1}$};
    \draw[line width=2pt,blue] (0,0)-- (0,1) node[left,midway] {$q^{-1}$};
    \draw[line width=0.7pt] (0,1)--(2,1)--(2,0);
    \draw[line width=0.7pt] (1,1)--(1,0);
    \draw[line width=0.7pt] (0,0)-- (1,0);
 %   \node at(1,-1.2) {$\tilde \G_t(q)$ for $t=\frac23$};
    		\draw (0.5, 0.5) node[] {$2$};
		\draw (1.5, 0.5) node[] {$3$};
		\draw (2.5, 0.5) node[] {$5$};
		\draw (2.5, 1.5) node[] {$7$};
		\draw (2.5, 2.5) node[] {$12$};
		\draw (3.5, 2.5) node[] {$17$};
		\draw (4.5, 2.5) node[] {$29$};
    \end{scope}
    \begin{scope}
  \draw[line width=2pt,red] (0,0)--(1,0) node[below,midway] {$q$};
    \draw[line width=0.7pt] (0,0)--(0,1);
    \draw[line width=2pt,red] (0,1)--(0,2) node[left,midway] {$q$};
    \draw[line width=2pt,blue] (0,2)--(0,3) node[left,midway] {$q^{-1}$};
    \draw[line width=0.7pt] (1,0)--(1,3);
    \draw[line width=0.7pt] (0,3)--(2,3);
    \draw[line width=0.7pt] (0,1)--(1,1);
    \draw[line width=0.7pt] (0,2)--(1,2);
    \draw[line width=2pt,blue] (1,2)--(2,2) node[below,midway] {$q^{-1}$};
    \draw[line width=0.7pt] (2,2)--(2,3);
   % \node at(1,-1.2) {$B$};
    \end{scope}
     \begin{scope}[shift={(2,2)}]
    \draw[line width=2pt,red] (0,0)--(1,0) node[below,midway] {$q$};
    \draw[line width=0.7pt] (0,0)--(0,1)--(1,1)--(1,0);
   % \node at(0.5,-1.2) {Final $B$};
    \end{scope}

   \end{tikzpicture}
\end{center}
\caption{Weighted graph $\tilde \G_t(q)$ for  $t=\frac23$.}
\label{fig:tilde G}
\end{figure}

\end{ex}

More generally, in the $q$-deformed matrices of $\PSL(2,\Z)$ the top right entry corresponds to a numerator of a $q$-rational \cite{MGOfmsigma}. 
The above observation  leads to another combinatorial model for the $q$-rationals.
The details will be developed elsewhere.

%%%%%%%%%%%%%%%%%%%%
\subsection{$F$-polynomials in the theory of cluster algebras}\label{CluSec}
%%%%%%%%%%%%%%%%%%%%

The theory of cluster algebras by Fomin and Zelevinsky 
gives an excellent framework to study
Markov numbers and the Markov equation.
Various work on Markov numbers and (generalized) Markov equation have been produced in this framework, see e.g. \cite{BaGy, BBH, EVW, LLRS, Alfredo, Pro}.

Here we  introduce a 3-parameter deformation of Cohn matrices related to the $F$-polynomials of the Markov cluster algebra.

First we introduce some ingredients of the theory of cluster algebras without going deeply into details.
We refer to \cite{FZ4, FWZ} for an overview and the main definitions.

The {\it Markov cluster algebra} is built from the following quiver 
%first observed by Propp~\cite{Pro}:
$$
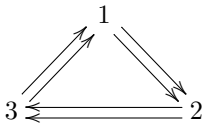
\xymatrix{
&1\ar@<2pt>@{->}[rd]\ar@<-2pt>@{->}[rd]\ar@<-2pt>@{<-}[ld]\ar@<2pt>@{<-}[ld]&\\
3&&2\ar@<-2pt>@{->}[ll]\ar@<2pt>@{->}[ll]
}
$$
by applying the standard mutation rules for cluster algebras. In this situation, the mutation rule can be presented as a branching rule in the infinite binary tree.

More precisely, the Markov cluster variables are elements in $\Z_{>0}[x_1^{\pm 1}, x_2^{\pm 1}, x_3^{\pm 1}]$ obtained in the binary tree from the initial triple $(x_1,x_2,x_3)$ and following the  rule:
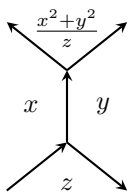
\begin{figure}[H]
\begin{center}
\begin{tikzpicture}[scale=0.8]
 \draw[thick,  ->, >=stealth] (0,0.2) -- (1,1);	
  \draw[thick,  ->, >=stealth] (1,1)-- (2,0.2);	
    \draw[thick,  ->, >=stealth] (1,1)-- (1,2.2);	
       \draw[thick,  ->, >=stealth] (1,2.2)-- (0,3);	
          \draw[thick,  ->, >=stealth] (1,2.2)-- (2,3);
      \node at (1,0.3) {$z$}; 
         \node at (0.4, 1.6) {$x$}; 
          \node at (1.6, 1.6) {$y$}; 
          \node at (1, 2.9) {$\frac{x^2+y^2}{z}$}; 
  \end{tikzpicture}
\caption{Branching rule for cluster variables (=mutation)}
\label{fig:fishCluster}
\end{center}
\end{figure}
We will hence label the Markov cluster variables with rational numbers and obtain the set $\{\X_t, \,t\in \Q\}$, so that $\X_{\frac01}=x_1$, $\X_{\frac10}=x_2$, $\X_{\frac{1}{-1}}=x_3$ are the initial cluster variables.
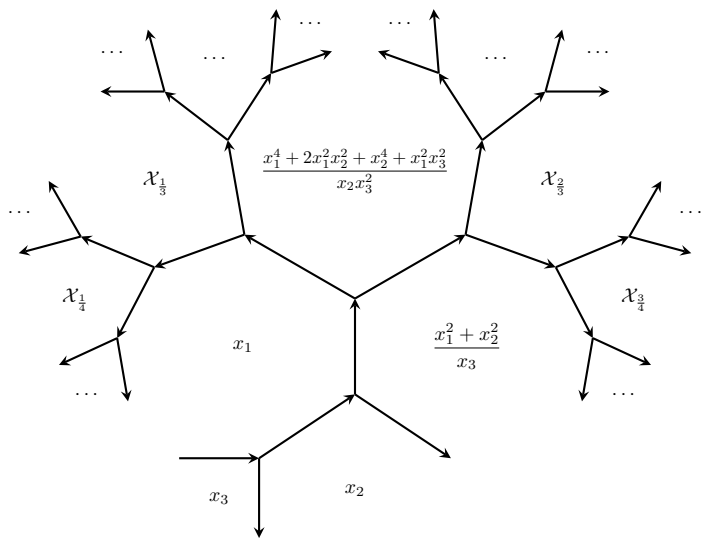
\begin{figure}[H]
\begin{center}
\begin{tikzpicture}[scale=0.85, every node/.style={scale=0.75}]
    \node at (1.75,1.25) {$\dfrac{x_1^2 + x_2^2}{x_3}$};
    \node at (-1.75,1.25){$x_1$};
      \draw (0,4) node[scale=0.9] {$\dfrac{x_1^4 + 2x_1^2x_2^2 + x_2^4 + x_1^2x_3^2}{x_2x_3^2}$};
     \node at (3.11, 3.8) {$\X_{\frac23}$};
    \node at (2.20, 5.76) {$\cdots$}; 
    \node at (4.37, 1.98) {$\X_{\frac34}$}; 
    \node at (0.7, 6.3) {$\cdots$}; 
    \node at (3.8, 5.85) {$\cdots$}; 
    \node at (5.25, 3.35) {$\cdots$}; 
    \node at (4.2, 0.5) {$\cdots$}; 
    \node at (-3.11, 3.8) {$\X_{\frac13}$};
    \node at (-2.20, 5.76) {$\cdots$}; 
    \node at (-4.37, 1.98) {$\X_{\frac14}$}; 
    \node at (-0.7, 6.3) {$\cdots$}; 
    \node at (-3.8, 5.85) {$\cdots$}; 
    \node at (-5.25, 3.35) {$\cdots$}; 
    \node at (-4.2, 0.5) {$\cdots$}; 
    
        \node at (0,-1){$x_2$};
            \node at (-2.125,-1.125){$x_3$};
 
%\draw[thick,  ->, >=stealth] (0, -0.5) -- (0, 0.75);
   % \draw[thick] (0, 0.75) -- (0, 2);
    \draw[thick, ->, >=stealth] (0, 0.5) -- (0, 2);
    
        \draw[thick,  ->, >=stealth] (-1.5, -0.5)-- (0, 0.5);
        
           \draw[thick,  ->, >=stealth] (0, 0.5) -- (1.5, -0.5);
           
          \draw[thick, <-, >=stealth] (-1.5,-0.5) -- (-2.75, -0.5);%green
        \draw[thick, ->, >=stealth] (-1.5,-0.5) -- (-1.5, -1.75);%green

    \draw[thick,  ->, >=stealth] (0,2) -- (1.73,3);
        
        \draw[thick,  ->, >=stealth] (1.73, 3) -- (1.99, 4.48);	

            \draw[thick,  ->, >=stealth] (1.99, 4.48) -- (1.31, 5.53);
            
                \draw[thick,  ->, >=stealth] (1.31, 5.53) -- (0.36, 5.87);
                \draw[thick,  ->, >=stealth] (1.31, 5.53) --  (1.23, 6.53);
                
            \draw[thick,  ->, >=stealth] (1.99, 4.48) -- (2.98, 5.24);	
            
                \draw[thick,  ->, >=stealth] (2.98, 5.24) -- (3.24, 6.21);
                \draw[thick,  ->, >=stealth] (2.98, 5.24) -- (3.98, 5.24);

        \draw[thick,  ->, >=stealth] (1.73, 3) -- (3.14, 2.49);

            \draw[thick,  ->, >=stealth] (3.14, 2.49) -- (4.29, 2.97);

                \draw[thick,  ->, >=stealth] (4.29, 2.97) -- (4.79, 3.84);
                \draw[thick,  ->, >=stealth] (4.29, 2.97) -- (5.26, 2.71);
            
            \draw[thick,  ->, >=stealth] (3.14, 2.49) -- (3.72, 1.38);
                
                \draw[thick,  ->, >=stealth] (3.72, 1.38) -- (4.63, 0.96);
                \draw[thick,  ->, >=stealth] (3.72, 1.38) -- (3.55, 0.39);
                
    \draw[thick,  ->, >=stealth] (0,2) -- (-1.73,3);
    
        \draw[thick,  ->, >=stealth] (-1.73, 3) -- (-1.99, 4.48);	
    
            \draw[thick,  ->, >=stealth] (-1.99, 4.48) -- (-1.31, 5.53);
    
                \draw[thick,  ->, >=stealth] (-1.31, 5.53) -- (-0.36, 5.87);
                \draw[thick,  ->, >=stealth] (-1.31, 5.53) -- (-1.23, 6.53);
    
            \draw[thick,  ->, >=stealth] (-1.99, 4.48) -- (-2.98, 5.24);	
    
                \draw[thick,  ->, >=stealth] (-2.98, 5.24) -- (-3.24, 6.21);
                \draw[thick,  ->, >=stealth] (-2.98, 5.24) -- (-3.98, 5.24);
    
        \draw[thick,  ->, >=stealth] (-1.73, 3) -- (-3.14, 2.49);
    
            \draw[thick,  ->, >=stealth] (-3.14, 2.49) -- (-4.29, 2.97);
    
                \draw[thick,  ->, >=stealth] (-4.29, 2.97) -- (-4.79, 3.84);
                \draw[thick,  ->, >=stealth] (-4.29, 2.97) -- (-5.26, 2.71);
    
            \draw[thick,  ->, >=stealth] (-3.14, 2.49) -- (-3.72, 1.38);
    
                \draw[thick,  ->, >=stealth] (-3.72, 1.38) -- (-4.63, 0.96);
                \draw[thick,  ->, >=stealth] (-3.72, 1.38) -- (-3.55, 0.39);
                
  \end{tikzpicture}
\caption{Tree of cluster variables}
\label{fig:Markov cluster}
\end{center}
\end{figure}

In a more general setting \cite{FZ4}, cluster variables are elements of $\Z_{>0}[y_1, y_2, y_3][x_1^{\pm 1}, x_2^{\pm 1}, x_3^{\pm 1}]$. 
In this setting, one recovers the cluster variables $\X_t$ described above by specializing $y_1=y_2=y_3=1$, and one obtains the associated $F$-polynomials $F_t$ as elements of $\Z_{>0}[y_1, y_2, y_3]$ by specializing  $x_1=x_2=x_3=1$. 

Specializing all the variables simultaneously to 1 one obtains the Markov numbers
$$
m^t=\X_t(1,1,1)=F_t(1,1,1).
$$
\begin{figure}[H]
\begin{center}
\begin{tikzpicture}[scale=0.85, every node/.style={scale=0.75}]
                \node at (1.75,1.25) {$y_3+1$};
    \node at (-1.75,1.25){$1$};
      \draw (0,4) node[scale=1] {$y_2y_3^2 + 2y_2y_3 + y_2 + 1$};
     \node at (3.11, 3.8) {$F_{\frac23}$};
    \node at (2.20, 5.76) {$\cdots$}; 
    \node at (4.37, 1.98) {$F_{\frac34}$}; 
    \node at (0.7, 6.3) {$\cdots$}; 
    \node at (3.8, 5.85) {$\cdots$}; 
    \node at (5.25, 3.35) {$\cdots$}; 
    \node at (4.2, 0.5) {$\cdots$}; 
    \node at (-3.11, 3.8) {$F_{\frac13}$};
    \node at (-2.20, 5.76) {$\cdots$}; 
    \node at (-4.37, 1.98) {$F_{\frac14}$}; 
    \node at (-0.7, 6.3) {$\cdots$}; 
    \node at (-3.8, 5.85) {$\cdots$}; 
    \node at (-5.25, 3.35) {$\cdots$}; 
    \node at (-4.2, 0.5) {$\cdots$}; 
    
        \node at (0,-1){$1$};
            \node at (-2.125,-1.125){$1$};

%\draw[thick,  ->, >=stealth] (0, -0.5) -- (0, 0.75);
   % \draw[thick] (0, 0.75) -- (0, 2);
    \draw[thick, ->, >=stealth] (0, 0.5) -- (0, 2);
    
        \draw[thick,  ->, >=stealth] (-1.5, -0.5)-- (0, 0.5);
        
           \draw[thick,  ->, >=stealth] (0, 0.5) -- (1.5, -0.5);
           
        \draw[thick, <-, >=stealth] (-1.5,-0.5) -- (-2.75, -0.5);%green
        \draw[thick, ->, >=stealth] (-1.5,-0.5) -- (-1.5, -1.75);%green

    \draw[thick,  ->, >=stealth] (0,2) -- (1.73,3);
        
        \draw[thick,  ->, >=stealth] (1.73, 3) -- (1.99, 4.48);	

            \draw[thick,  ->, >=stealth] (1.99, 4.48) -- (1.31, 5.53);
            
                \draw[thick,  ->, >=stealth] (1.31, 5.53) -- (0.36, 5.87);
                \draw[thick,  ->, >=stealth] (1.31, 5.53) --  (1.23, 6.53);
                
            \draw[thick,  ->, >=stealth] (1.99, 4.48) -- (2.98, 5.24);	
            
                \draw[thick,  ->, >=stealth] (2.98, 5.24) -- (3.24, 6.21);
                \draw[thick,  ->, >=stealth] (2.98, 5.24) -- (3.98, 5.24);

        \draw[thick,  ->, >=stealth] (1.73, 3) -- (3.14, 2.49);

            \draw[thick,  ->, >=stealth] (3.14, 2.49) -- (4.29, 2.97);

                \draw[thick,  ->, >=stealth] (4.29, 2.97) -- (4.79, 3.84);
                \draw[thick,  ->, >=stealth] (4.29, 2.97) -- (5.26, 2.71);
            
            \draw[thick,  ->, >=stealth] (3.14, 2.49) -- (3.72, 1.38);
                
                \draw[thick,  ->, >=stealth] (3.72, 1.38) -- (4.63, 0.96);
                \draw[thick,  ->, >=stealth] (3.72, 1.38) -- (3.55, 0.39);
                
    \draw[thick,  ->, >=stealth] (0,2) -- (-1.73,3);
    
        \draw[thick,  ->, >=stealth] (-1.73, 3) -- (-1.99, 4.48);	
    
            \draw[thick,  ->, >=stealth] (-1.99, 4.48) -- (-1.31, 5.53);
    
                \draw[thick,  ->, >=stealth] (-1.31, 5.53) -- (-0.36, 5.87);
                \draw[thick,  ->, >=stealth] (-1.31, 5.53) -- (-1.23, 6.53);
    
            \draw[thick,  ->, >=stealth] (-1.99, 4.48) -- (-2.98, 5.24);	
    
                \draw[thick,  ->, >=stealth] (-2.98, 5.24) -- (-3.24, 6.21);
                \draw[thick,  ->, >=stealth] (-2.98, 5.24) -- (-3.98, 5.24);
    
        \draw[thick,  ->, >=stealth] (-1.73, 3) -- (-3.14, 2.49);
    
            \draw[thick,  ->, >=stealth] (-3.14, 2.49) -- (-4.29, 2.97);
    
                \draw[thick,  ->, >=stealth] (-4.29, 2.97) -- (-4.79, 3.84);
                \draw[thick,  ->, >=stealth] (-4.29, 2.97) -- (-5.26, 2.71);
    
            \draw[thick,  ->, >=stealth] (-3.14, 2.49) -- (-3.72, 1.38);
    
                \draw[thick,  ->, >=stealth] (-3.72, 1.38) -- (-4.63, 0.96);
                \draw[thick,  ->, >=stealth] (-3.72, 1.38) -- (-3.55, 0.39);

\end{tikzpicture}
\caption{Tree of $F$-polynomials.}
\label{fig:Markov Fpoly}
\end{center}
\end{figure}
\begin{ex}
(a) Further examples of Markov cluster variables are as follows.
\begin{eqnarray*}
\X_{\frac13} &=& \dfrac{x_1^6 + 3x_1^4x_2^2 + 3x_1^2x_2^4 + x_2^6 + (2x_1^4 + 2x_1^2x_2^2)x_3^2 + x_1^2x_3^4}{x_2^2x_3^3},\\[4pt]
\X_{\frac14} &=& \dfrac{(x_1^8 + 4x_1^6x_2^2 + 6x_1^4x_2^4 + 4x_1^2x_2^6 + x_2^8 + (3x_1^6 + 6x_1^4x_2^2 + 3x_1^2x_2^4)x_3^2 + (3x_1^4 + 2x_1^2x_2^2)x_3^4 + x_1^2x_3^6}{x_2^3x_3^4},\\[4pt]
\X_{\frac23} &=& \dfrac{x_1^8 + 4x_1^6x_2^2 + 6x_1^4x_2^4 + 4x_1^2x_2^6 + x_2^8 + (2x_1^6 + 5x_1^4x_2^2 + 4x_1^2x_2^4 + x_2^6)x_3^2 + x_1^4x_3^4}{x_1x_2^2x_3^4},\\[4pt]
\X_{\frac34}&=& \dfrac{x_1^{12} + 6x_1^{10}x_2^2 + 15x_1^8x_2^4 + 20x_1^6x_2^6 + 15x_1^4x_2^8 + 6x_1^2x_2^{10} + x_2^{12}}{x_1^2x_2^3x_3^6}\\
&&+\; \dfrac{(3x_1^{10} + 14x_1^8x_2^2 + 26x_1^6x_2^4 + 24x_1^4x_2^6 + 11x_1^2x_2^8 + 2x_2^{10})x_3^2}{x_1^2x_2^3x_3^6}\\
&&+\; \dfrac{(3x_1^8 + 8x_1^6x_2^2 + 8x_1^4x_2^4 + 4x_1^2x_2^6 + x_2^8)x_3^4 + x_1^6x_3^6}{x_1^2x_2^3x_3^6}.
\end{eqnarray*}

(b) 
The corresponding $F$-polynomials are the following
\begin{eqnarray*}
F_{\frac13} &=& y_2^2y_3^3 + 3y_2^2y_3^2 + 3y_2^2y_3 + y_2^2 + 2y_2y_3 + 2y_2 + 1,\\[4pt]
F_{\frac14} &=& y_2^3y_3^4 + 4y_2^3y_3^3 + 6y_2^3y_3^2 + 4y_2^3y_3 + 3y_2^2y_3^2 + y_2^3 + 6y_2^2y_3 + 3y_2^2 + 2y_2y_3 + 3y_2 + 1,\\[4pt]
F_{\frac23} &=& y_1y_2^2y_3^4 + 2y_1y_2^2y_3^3 + y_2^2y_3^4 + y_1y_2^2y_3^2 + 4y_2^2y_3^3 + 6y_2^2y_3^2 + 4y_2^2y_3 + 2y_2y_3^2 + y_2^2 + 4y_2y_3 + 2y_2 + 1,\\[4pt]
F_{\frac34} &=& y_1^2y_2^3y_3^6 + 2y_1^2y_2^3y_3^5 + 2y_1y_2^3y_3^6 + y_1^2y_2^3y_3^4 + 8y_1y_2^3y_3^5 + y_2^3y_3^6 + 12y_1y_2^3y_3^4 + 6y_2^3y_3^5 + 8y_1y_2^3y_3^3\\
&&+\;  2y_1y_2^2y_3^4 + 15y_2^3y_3^4 + 2y_1y_2^3y_3^2 + 4y_1y_2^2y_3^3 + 20y_2^3y_3^3 + 3y_2^2y_3^4 + 2y_1y_2^2y_3^2 + 15y_2^3y_3^2 \\
&&+\;  12y_2^2y_3^3 + 6y_2^3y_3 + 18y_2^2y_3^2 + y_2^3 + 12y_2^2y_3 + 3y_2y_3^2 + 3y_2^2 + 6y_2y_3 + 3y_2 + 1
\end{eqnarray*}
\end{ex}

As recalled in the previous sections the Markov number $m^t$ appears as the top right element of the corresponding Cohn matrix.
We introduce a deformation of the Cohn matrices using the parameters $(y_1,y_2,y_3)$ leading to a generalization of this fact.
Consider the $3$-parameter quantisation governed by
\begin{eqnarray*}
   \hat A &=& 
   \begin{pmatrix}
        y_3 & 1 \\
        0 & 1
    \end{pmatrix}
    \begin{pmatrix}
        y_2 & 0 \\
        y_2 & 1
    \end{pmatrix} \\
    &=&
     \begin{pmatrix}
        y_2 + y_2 y_3 & 1 \\
        y_2 & 1
    \end{pmatrix}.
\\[6pt]
   \hat B &=& \begin{pmatrix}
        y_3 & 1 \\
        0 & 1
    \end{pmatrix}
    \begin{pmatrix}
        y_1 & 1 \\
        0 & 1
    \end{pmatrix}
    \begin{pmatrix}
        y_2 & 0 \\
        y_2 & 1
    \end{pmatrix}
    \begin{pmatrix}
        y_3 & 0 \\
        y_2 & 1
    \end{pmatrix} \\[2pt]
    &=& \begin{pmatrix}
        y_1 y_2 y_3^2 + y_1 y_2^2 + 2 y_1 y_2 + y_2 & y_3 + 1 \\[4pt]
        y_2 y_3 + y_2 & 1
    \end{pmatrix}.
\end{eqnarray*}
Note that these are product of what we can consider as $y$-deformed versions of two standard generators $T$ and $L$ of $\SL_2(\mathbb{Z})$.
Specializing $y_1=y_2=y_3=q$ in $\hat A, \hat B$ will return the $q$-deformed Cohn matrices $A(2)_q, B(2)_q$,  up to a power of $q$.

\begin{conj}
For every rational $t$, corresponding to the Christoffel word $\tilde w_t (A,B)$, the deformed matrix $\hat C_t$ given by $\tilde w_t (\hat A, \hat B)$ has a top right entry equals to the $F$-polynomial 
$F_t$, up to a unit element of the ring $\Z_{}[y_1, y_2, y_3]$. 
In particular for $0\leq t \leq 1$ the top right entry is exactly $F_t$.
\end{conj}

Computer experiments are the evidence for this conjecture. For instance
$$
\hat C_{\frac12}=\hat A \hat B= \begin{pmatrix}(y_1y_2y_3^2 + y_2y_3^2 + 2y_2y_3 + y_2 + 1)y_2(y_3 + 1)          &                  y_2y_3^2 + 2y_2y_3 + y_2 + 1\\[4pt]
  (y_1y_2y_3^2 + y_2y_3^2 + 2y_2y_3 + y_2 + y_3 + 1)y_2             &                           y_2y_3 + y_2 + 1

\end{pmatrix},
$$

$$
\hat C_{\frac10}={\hat A}^{-1} \hat B= \begin{pmatrix}y_1y_3 + y_3 + 1        &                 y_2^{-1} \\[4pt]
 -y_1y_2y_3          &                        0
\end{pmatrix}.
$$

\bigskip

\noindent{\bf Acknowledgments}.
We are grateful to Ralf Schiffler and Alexander Veselov for enlightening discussions.
This collaborative work received the support of a SSHN Fellowship from the French Embassy in the UK.

\bigskip

\noindent{\bf Conflict of interest statement}.
On behalf of all authors, the corresponding author states that there is no conflict of interest.

%%%%%%%%%%%%%%%%%%%%

\bigskip 

\bigskip  

\noindent
{\sc
{Sam Evans,
Department of Mathematical Sciences, Loughborough University, Loughborough LE11
3TU, UK},\\
{Email : S.J.Evans@lboro.ac.uk}

\medskip

\noindent
{Perrine Jouteur,
Universit\'e de Reims Champagne Ardenne,
Laboratoire de Math\'e\-ma\-tiques, CNRS UMR9008,
Moulin de la Housse - BP 1039,
51687 Reims cedex 2,
France
}\\
{Email: perrine.jouteur@univ-reims.fr}

\medskip

\noindent
{Sophie Morier-Genoud,
Universit\'e de Reims Champagne Ardenne,
Laboratoire de Math\'e\-ma\-tiques, CNRS UMR9008,
Moulin de la Housse - BP 1039,
51687 Reims cedex 2,
France,
}\\
{Email:  sophie.morier-genoud@univ-reims.fr}

\medskip

\noindent
{Valentin Ovsienko,
Centre National de la Recherche Scientifique,
Laboratoire de Math\'e\-ma\-tiques,
Universit\'e de Reims Champagne Ardenne,
Moulin de la Housse - BP 1039,
51687 Reims cedex 2,
France},\\
{Email:  valentin.ovsienko@univ-reims.fr}
}

\end{document}